\documentclass[11pt]{article}
\usepackage[utf8]{inputenc}
\usepackage{amssymb,amsopn,amsmath,amsthm,eucal,hyperref, cleveref,lipsum,scrextend, mathrsfs,array,graphicx,xcolor,bbm,enumerate,wasysym,mathtools}
\usepackage{authblk}

\usepackage{verbatim}
\usepackage[all]{xy}
\usepackage[displaymath, mathlines]{lineno}
\newtheorem{thm}{Theorem}[section]
\newtheorem{lemma}[thm]{Lemma}

\newtheorem{defn}[thm]{Definition}

\newtheorem{prop}[thm]{Proposition}
\newtheorem{cor}[thm]{Corollary}

\theoremstyle{definition}
\newtheorem{rem}[thm]{Remark}

\let \le \leqslant
\let \leq \leqslant

\let \ge \geqslant
\let \geq \geqslant

\let \setminus \smallsetminus

\let\OLDthebibliography\thebibliography
\renewcommand\thebibliography[1]{
  \OLDthebibliography{#1}
  \setlength{\parskip}{0pt}
  \setlength{\itemsep}{0pt plus 0.3ex}
}

\usepackage[activate={true,nocompatibility},final,tracking=true,kerning=true,spacing=true,factor=1100,stretch=10,shrink=10]{microtype}
\microtypecontext{spacing=nonfrench}

\newcommand{\Q}{\mathbb Q}

\newcommand{\R}{\mathbb R}
\newcommand{\C}{\mathbb C}
\renewcommand{\H}{\mathbb H}

\newcommand{\N}{\mathbb N}

\newcommand{\D}{\mathbb D}
\newcommand{\E}{\mathbb E}

\renewcommand{\P}{\mathbb P}
\renewcommand{\1}{\mathbf 1}

\newcommand{\A}{\mathbb A}
\newcommand{\Aa}{\mathbb{A}_{-a,a}}

\newcommand{\MM}{\mathfrak{M}}
\newcommand{\JJ}{\mathcal J}
\newcommand{\V}{\mathcal V}
\newcommand{\M}{\mathbb M}
\newcommand{\Os}{\mathcal O}
\providecommand{\dimh}{\textup{dim}_H}

\newcommand{\F}{\mathcal F}
\newcommand{\Fa}{\mathcal{F}_{\mathbb{A}_{-a,a}}}

\newcommand{\eps}{\epsilon}

\newcommand{\ave}[1]{\textcolor{red}{#1}}
\let \epsilon \varepsilon

\DeclareMathOperator{\SLE}{SLE}

\title{Dimension of two-valued sets via imaginary chaos}

\author[1]{Lukas Schoug}
\author[2]{Avelio Sep\'ulveda}
\author[3]{Fredrik Viklund}
\affil[1,3]{\it KTH Royal Institute of Technology}
\affil[2]{\it Universit\'e Lyon 1}
\date{}

\begin{document}

\maketitle
\begin{abstract}
    Two-valued sets are local sets of the two-dimensional Gaussian free field (GFF) that can be thought of as representing all points of the domain that may be connected to the boundary by a curve on which the GFF takes values only in $[-a,b]$. Two-valued sets exist whenever $a+b\geq 2\lambda,$ where $\lambda$ depends explicitly on the normalization of the GFF. We prove that the almost sure Hausdorff dimension of the two-valued set $\A_{-a,b}$ equals $d=2-2\lambda^2/(a+b)^2$. For the two-point estimate, we use the real part of a ``vertex field'' built from the purely imaginary Gaussian multiplicative chaos. We also construct a non-trivial $d$-dimensional measure supported on $\A_{-a,b}$ and discuss its relation
	with the $d$-dimensional conformal Minkowski content for $\A_{-a,b}$.


\end{abstract}
 \section{Introduction}
		Let $\Gamma$ be a two-dimensional  Gaussian free field (GFF) in the unit disc $\mathbb{D}$ with Dirichlet boundary condition. 
	A local set  for $\Gamma$ is a random set $A \subset \overline{\mathbb{D}}$ coupled with $\Gamma$ and a field $\Gamma_A$, harmonic on $\mathbb{D} \setminus A$, such that conditionally on the pair $(\Gamma,A)$, the field $\Gamma - \Gamma_A$ is a GFF in $\mathbb{D} \setminus A$. Local sets for Markov random fields were first studied in the 1980s by Rozanov \cite{Roz} and later rediscovered in the context of the 2D GFF by Schramm and Sheffield \cite{SchSh2}. Well-known and important examples of local sets are SLE$_\kappa$ curves coupled with the GFF in the sense of \cite{MS1}. 
	
		For $a+b \ge \lambda = \pi/2$ with our normalization of the field, it is possible to construct a local set, $A$, with the property that the associated field $\Gamma_A$ can be represented by a function in $D\setminus A$ that only takes the two values $-a$ and $b$. This is the \emph{two-valued local set}\footnote{We will review the relevant definitions in detail below.} $\A_{-a,b}$. One way to think about $\A_{-a,b}$ is as a generalization to two-dimensional `time' of the first exit time of the interval $[-a,b]$ for a 1D Brownian motion started at $0$. 
		
		The class of two-valued local sets was introduced and studied systematically in \cite{ASW}, see also, e.g., \cite{AS,ALS1, ALS2}. The conformal loop-ensemble with parameter $4$, CLE$_4$ (choosing $a=b=2\lambda$), the arc-loop ensemble used to couple the free-boundary GFF with the GFF with zero boundary condition  (choosing $a=b=\lambda$) \cite{QW} and, conjecturally, the limit of cluster interfaces in the XOR-Ising model  (choosing $a=b=2\sqrt{2}\lambda$) \cite{wilson} are all two-valued local sets. The computation of the expectation dimension of $\A_{-a,b}$ was sketched in \cite{ASW} and our first theorem gives the almost sure result for the Hausdorff dimension.

	
	\begin{thm}\label{thm:dim}
	 Let $\Gamma$ be a GFF in $\mathbb{D}$. For all $a,b>0$ such that $a+b\geq 2\lambda$, almost surely,
	\[
	\dimh \, \A_{-a,b}= 2-\frac{2\lambda^2}{(a+b)^2}.\]
	\end{thm}
The proof of Theorem~\ref{thm:dim} will be completed at the end of Section~\ref{sect:dim}. We will write
$d = 2-2\lambda^2/(a+b)^2$ throughout the paper.

By setting $a=b=2\lambda$ we recover the well-known result that the CLE$_4$ carpet dimension equals $15/8$. Moreover, since for every $a>0$,  the \emph{first-passage set} of level $-a$ (see \cite{ALS1}) can be constructed as $\A_{-a}=\overline{\bigcup_{b'} \A_{-a,b'}}$, we have the following corollary, first proved in \cite{ALS1}.

\begin{cor}
Fix $a>0$. Then $\dimh \, \A_{-a}=2$ almost surely.
\end{cor}
As usual, the main difficulty is the correlation estimate. In the case at hand there are several possible approaches to it, e.g., using SLE techniques  and loop measures, see \cite{MW} and \cite{NW} respectively, for related work. The latter approach gives very short proofs in the special case when $a,b$ are integer multiples of $2\lambda$ but more work is needed for general $a,b$. 

In this paper, we will follow a different path and prove the correlation estimate using an observable constructed from the purely imaginary Gaussian multiplicative chaos \cite{LRV,JSW} and we feel this approach is of some independent interest. The imaginary chaos $\mathcal{V}^{i\sigma}=\mathcal{V}^{i\sigma}(\Gamma)$ is the field that one gets in the limit as $\eps \to 0+$ of $\eps^{-\sigma^2/2} e^{i\sigma\Gamma_{\eps}(z)}$, where $\sigma$ is real and $\Gamma_\eps(z)$ are circle averages of $\Gamma$. If $\sigma \in (0,\sqrt{2})$ one may take this limit in probability in the Sobolev space $H^{s}, s < -1$, see \cite{JSW} and further discussion below. In the language of conformal field theory, $\mathcal{V}^{i\sigma}$ is a vertex field (or operator) with imaginary charge $i\sigma$, see, e.g., \cite{KM}. In Section~\ref{intuition} we will try to provide some intuition for why the imaginary chaos carries information about the geometry of the GFF. 

The construction of two-valued sets, which we will review below, uses $\SLE_4$-type processes. So an immediate lower bound on the dimension is that of $\SLE_4$, namely $3/2$ \ave{\cite{Bef}}, which turns out to be the dimension of the smallest two-valued local set: the arc-loop ensemble, ALE, gotten by setting $a=b=\lambda$.  Actually, our approach can be used to give a short proof of correlation estimate for SLE$_4$.

	 It is shown in  \cite{JSW} that the renormalized scaling limit of the spin configuration of an XOR-Ising model with +/+ boundary condition agrees in law (up to a constant and with our normalisation) with the real part of an imaginary chaos with $\sigma=1/\sqrt{2}$. In fact, ideas appearing in this paper can be combined with work in \cite{Sep2} in order to shed light on some aspects of Wilson's conjecture about interfaces of the XOR-Ising model.

In Section~\ref{sect:minkowski}, we use the imaginary chaos to construct a non-trivial measure supported on $\A_{-a,b}$. This measure represents $\mathcal{V}^{i\sigma}(\Gamma-(a-b)/2)$ conditional on the values of the GFF on top of $\A_{-a,b}$. We do not prove that the $d$-dimensional conformal Minkowski content on $\A_{-a,b}$ exists, but we show that on the event that it does then it agrees with the measure we construct up to a multiplicative constant. The following theorem summarizes the results of Section~\ref{sect:minkowski}.
\begin{thm}
Fix $a,b$ such that $a+b \ge 2\lambda$ and set $\sigma_c=2\lambda/(a+b)$. For $\delta > 0$, define a random measure on $\mathbb{D}$ by the relation
\[
d\mu_\delta =\delta  r_{\D \setminus \A_{-a,b}}(z)^{-(\sigma_c-\delta)^2/2} dz,
\]
where $r_{\D \setminus \A_{-a,b}}(z)$ is the conformal radius at $z$ of the connected component of $\D \setminus \A_{-a,b}$ containing $z$. Then as $\delta \to 0+$, $\mu_\delta$ converges in law with respect to the weak topology to a random measure $\mu$ supported on $\A_{-a,b}$ and such that
\[
\mathbb{E}\left[\mu(\D) \right] = \frac{2}{a+b}\sin\left(\frac{\pi a}{a+b}\right) \int_\D r_\D(z)^{d-2} dz.
\]
Moreover, there is a deterministic constant $c$ depending only on $a$ and $b$ such that on the event that the $d$-dimensional conformal Minkowski content of $\A_{-a,b}$ exists, then it is necessarily equal to the measure $c\mu$. 
\end{thm}
Here and below we write $dz$ for two-dimensional Lebesgue measure.
	 
\subsection{Imaginary chaos and the geometry of the GFF}\label{intuition}To give some intuition as to why imaginary chaos may encode geometric information about two-valued sets let us discuss a few related examples. The first one is elementary. Suppose $a,\sigma>0$ and that $B_t$ is  linear Brownian motion started from $x$ with $|x|<a$. If we let $\tau_{-a,a}$ be the first exit time of the interval $(-a,a)$,  then it is not hard to see that
		\[V_t^{i\sigma}:=\exp\left( i\sigma B_{t\wedge \tau_{-a,a}} + \frac{\sigma^2}{2} (t\wedge \tau_{-a,a})\right) \]
		is a uniformly integrable martingale if and only if $a\sigma  < \pi/2$. Thus as long as $\sigma < \pi/2a=:\sigma_c$, we have that for all $x\in[-a,a]$
		\begin{equation}\label{e. laplace transform}
		\E^x\left[ e^{\frac{\sigma^2}{2} \tau_{-a,a}} \right] = \frac{\cos(\sigma x)}{\cos(\sigma a)}.
		\end{equation}
		Hence, we can obtain  the Laplace transform of $\tau_{-a,a}$ by analyzing (the real part of) $V^{i\sigma}_t$. Furthermore, the critical value $\sigma_c$ gives the pole with largest real part and so the tail behavior of the distribution. Actually, this computation essentially gives the one-point estimate for the two-valued set, see Lemma~\ref{l.1point}.

Let us come back to the heuristics for the proof of Theorem \ref{thm:dim}. It is enough to consider the symmetric case $a=b$. To obtain the correlation estimate we study the conditional expectation of the real part of the imaginary chaos close to a given point, that is, the conditional law given $\Aa$ of random variables of the form $C_z=\textup{Re}(\mathcal{V}^{i\sigma}, \1_{B(z,\delta)})$, where $B(z,\delta)$ is the ball of radius $\delta$ about $z$. For a fixed point the expected value of this quantity is small. However, for the exceptional points that are close to $\A_{-a,a}$, the conditional expectation becomes large due to a factor involving the conformal radius of $\D \setminus \Aa$ to a negative power and since we take the real part, the factor coming from the harmonic function, which takes values $\pm a$, is in fact constant. As in the case of Brownian motion, the growth rate near $\A_{-a,a}$ will depend on the parameter $\sigma$, and the desired bound matching the one-point estimate corresponds the `critical' $\sigma=\sigma_c$. Because of the good correlation structure of the two-valued set it possible to estimate the behavior of an appropriate two-point function with two `insertions' and use it to prove the two-point estimate. 

Let us make a few further remarks.

An important quantity in the study of the fractal geometry of SLE curves is the SLE$_\kappa$ Green's function, i.e., the renormalised limit of the probability that the SLE$_\kappa$ path gets near a given point. Conditioning on a portion of the path, this Green's function gives a local SLE$_\kappa$ martingale which blows up on paths that get near the marked point. In some sense, the conditional expectation $\mathbb{E}[\textrm{Re} \, (\mathcal{V}^{i\sigma_c}, f) \mid \A_{-a,a} ]$ with $f$ a point mass at $z$ is the analogue of this martingale when one considers the whole path.  


Actually, the SLE Green's function and several other geometric SLE observables can (at least formally) be represented as CFT vertex fields \cite{KM}: roughly speaking, given a simply connected domain $D$ with marked distinct boundary points $a,b$, one considers fields of the form
\[
O^{\sigma, \sigma_*} = O^{\sigma, \sigma_*}_{D,a,b}= X \cdot e^{i \sigma \Gamma^+  - i \sigma_* \overline{\Gamma^+}},
\]
where $X$ is an explicit deterministic function depending on the configuration $(D,a,b)$ and $\Gamma^+$ is a formal multivalued object known as chiral field associated to the GFF,  given as $\Gamma^+ = \int^z J dz$, where $J=\partial \Gamma$ is the GFF current, see Lecture~9 of \cite{KM}.
For example, if $\sigma = D_\kappa/2$ then the SLE Green's function can be represented as the correlation function $G(z) = \langle O^{\sigma,\sigma}(z) \rangle$ for a GFF with particular boundary data. However, as opposed to the imaginary chaos, the precise probabilistic meaning of the chiral vertex field is not clear.

\subsection*{Acknowledgements}
Schoug was supported by the Knut and Alice Wallenberg Foundation. Viklund was supported by the Knut and Alice Wallenberg Foundation, the Swedish Research Council, and the Gustafsson foundation.
	\section{Preliminaries}
	
	\subsection{Gaussian free field and local sets}
Let $D \subsetneq \mathbb{C}$ be a simply connected domain and let $G(z,w)=G_D(z,w)$ be the Brownian motion Green's function for $D$ with Dirichlet boundary condition. Recall that if we fix $w \in D$ and let $z \mapsto u(z) = u(z,w)$ solve the Dirichlet problem for $D$ with $u(\zeta)=\log|\zeta-w|, \zeta \in \partial D,$ as boundary data, then $G(z,w) = G(w,z)= u(z,w) - \log|z-w|$. 

The Dirichlet energy space $\mathcal{E}=\mathcal{E}_D$ is the completion of $C_0^\infty(D)$ using the norm \[\|f\|_{\mathcal{E}}^2 = \int_{D \times D}  f(z)f(w) G(z,w) dz dw.\]
The (real) Gaussian free field  on $D$, $\Gamma:\mathcal{E} \to L^2(\mathbb{P})$, is the Gaussian process (or from a different point of view, the Gaussian Hilbert space)  indexed by the Dirichlet energy space $\mathcal{E}$ with correlation kernel given by the Green's function. That is, a collection of mean zero Gaussian random variables $\Gamma(f), f \in \mathcal{E},$ such that
\[
\mathbb{E}[ \Gamma(f) \Gamma(g)] = \int_{D \times D} f(z) g(w) G(z,w) dzdw. 
\] 

For $\epsilon >0$ one can also realize the GFF as a random element of the Sobolev space $H^{-\epsilon}(D)$, i.e., a random distribution such that for each test function $f$, $\Gamma(f)$ is a centered Gaussian random variable as above.  

The last paragraph defines the GFF with Dirichlet boundary condition. We may impose other deterministic boundary conditions by adding to $\Gamma$ the solution to the Dirichlet problem with the desired boundary data.


We say that $A$ is a \emph{local set} for the GFF $\Gamma$ if $A$ is a random subset\footnote{A random subset of $\overline{D}$ is by definition a random element of the space of compact subsets of $\overline{D}$ with the Borel sigma algebra and topology generated by Hausdorff distance.} of $\overline{D}$  with the property that there is a coupling of $\Gamma$, $A$ and a field $\Gamma_A$, where:
\begin{itemize}
\item{$\Gamma_A$ can be represented by a harmonic function $h_A$ on $D \setminus A$;}
\item{Conditionally on the tuple $(\Gamma,A)$, the random distribution $\Gamma^A := \Gamma - \Gamma_A$ is a GFF in $D \setminus A$.}
\end{itemize} 
\begin{defn}\label{defn:sigma algebra}
Let $A$ be a local set coupled with a GFF $\Gamma$. We denote $\F_A$ the sigma algebra generated by the pair $(A,\Gamma_A)$.
\end{defn}

We say that a local set $A$ such that $h_A$ is bounded\footnote{The definition in the case where $h_A(\cdot)$ does not belong to $L_{loc}^1$ is discused in \cite{Sep}.} is \emph{thin} if for every test function $f \in C^\infty_0(D)$, \[(\Gamma_A,f) = \int_{D\setminus A} h_A(z) f(z) dz.\] That is, the field $\Gamma_A$ does not ``charge'' $A$. The following sufficient condition to be thin concerns the size of the set $A$ and can be found in \cite{Sep}.
\begin{prop}[Proposition 1.3 of \cite{Sep}]\label{p.thin} Let $A$ be a local set of a GFF $\Gamma$ such that its upper Minkowski dimension is a.s. strictly smaller than $2$. Then $A$ is thin.
\end{prop}

	\subsection{Level lines}
	To construct two-valued sets for a GFF $\Gamma$, we need random curves $\eta$ such that for any stopping time $\tau$ of $\eta$ the set $\eta_\tau:=\eta([0,\tau])$ is a thin local set of $\Gamma$ and such that $h_{\eta_\tau}$ is bounded. The first example of a curve like this was found by Schramm and Sheffield in \cite{SchSh2} and then further expanded in \cite{WaWu,PW} using the techniques of \cite{MS1}. Here is the special case needed for this paper.
	
	\begin{thm}[Theorem 1.1.1 and 1.1.2 of \cite{WaWu}]\label{thm:level-line}
	 Let $\rho^L,\rho^R> -2$ and let $\Gamma$ be a GFF in $\H$ with boundary condition $-\lambda(1+\rho^L)$ on $\R_-$ and $-\lambda(1+\rho^R)$ on $\R_+$. Then there exists a random continuous curve $\eta$ such that for all stopping times $\tau$ the set $\eta_\tau=\eta[0, \tau]$ is a thin local set of $\Gamma$ such that $h_{\eta_\tau}$ is the unique bounded harmonic function in $\H\setminus \eta_\tau$ with boundary condition 
	 \begin{align*}
	     \left\{ \begin{array}{l l}
	   -\lambda(1+\rho^L) & \text{on } \R_-;\\
	    \lambda(1+\rho^R)  & \text{on } \R_+;\\
	    -\lambda& \text{on the left-hand side of } \eta_\tau;\\
	    \lambda& \text{on the right-hand side of } \eta_\tau.
	    \end{array} \right.
	 \end{align*}
	  Furthermore, the curve $\eta$ is a deterministic function of $\Gamma$ and the law of $\eta$ is that of SLE$_4(\rho^L,\rho^R)$.
	\end{thm}
	In the statement, the left-hand side of the curve $\gamma_\tau$ is defined as those prime-ends on the trace that the uniformizing Loewner map $g_\tau$ maps to the left of $g_\tau(\eta(\tau))$ and similarly for the right-hand side.
	
	When the SLE $\eta$ is coupled with $\Gamma$ in the sense of Theorem~\ref{thm:level-line}, we say that it is a \emph{level line} of the $\Gamma$.

	
	\subsection{Two-valued local sets}
	
	Fix $a,b>0$, and let $\Gamma$ be a zero boundary GFF in a simply connected domain $D$. We say that $\A_{-a,b}$  is a two-valued set (abbreviated TVS) of levels $-a$ and $b$ if it is a thin local set of $\Gamma$ with the property \hypertarget{tvs}that: 
	\begin{itemize}
		\item[(\twonotes)] For all $z\in D\setminus \A_{-a,b}$, a.s. $h_{\A_{-a,b}} (z) \in \{-a,b\}$.
	\end{itemize} 
	
	We denote by $r_D(z)$ the conformal radius of $D$ at $z$. Let us recall the main properties of two-valued sets.
	\begin {prop}[Proposition 2 of \cite{ASW}]
	\label {cledesc2}
	Suppose $-a < 0 < b$. 
	There exists a thin local set $\A_{-a,b}$ coupled with a GFF $\Gamma$ satisfying \hyperlink{tvs}{(\twonotes)} if and only if $a+b \ge 2 \lambda$. Moreover, $\A_{-a,b}$ satisfies the following properties:
	\begin{enumerate}		\item If $A'$ is a thin local set of $\Gamma$ satisfying \hyperlink{tvs}{(\twonotes)}, then $A' = \A_{-a,b}$ almost surely.  
		\item The local sets $\A_{-a,b}$ are deterministic functions of $\Gamma$.
		\item If $[-a,b] \subset [-a', b']$ and $-a < 0 < b$ with $b+a \ge 2\lambda$,  then almost surely, $\A_{-a,b} \subset \A_{-a', b'}$. 
		\item For $z\in D$ fixed, the random variable $\log r_D(z)-\log r_{D \setminus \A_{-a,b}}(z)$ is distributed as the first hitting time of $\{-\pi a/2\lambda,\pi b/2\lambda\}$ by a one-dimensional Brownian motion started from $0$.
	\end{enumerate}
	\end {prop} 
We also mention that being a bounded-type thin local set (a BTLS), we have that $\A_{-a,b} \cup \partial D$ is a connected set, see \cite{ASW} for further properties.	

\begin{rem}\label{rem:multifractality}
Note that, a priori, if one fixes an instance of a GFF $\Gamma$, one can only define $\A_{-a,b}$ simultaneously for a countable subset of $(a,b)\in \R_+^2$. However, using the monotonicity property we can give a definition on a probability one event for all $a,b$ such that $a+b> 2\lambda$ simultaneously. Indeed, take $a,b\in \R^+$ such that $a+b> 2\lambda$. Then we have
\[ \A_{-a,b}:=\overline{\bigcup_{\substack{a',b'\in \Q^+\\a'+b'\geq 2\lambda\\ a'\leq a, b'\leq b}} \A_{-a',b'}}.\]
This fact follows from the uniqueness of two-valued sets (Proposition \ref{cledesc2}) and Lemma 2.3 of \cite{ALS1} and we have that this defines $\A_{-a,b}$ simultaneously for all $a,b\in \R^+$ with $a+b > 2\lambda$. Furthermore, by the monotonicity of $\A_{-a,b}$ (Proposition \ref{cledesc2}), if we prove Theorem \ref{thm:dim}, we obtain immediately that, almost surely, for all $a+b > 2\lambda$
\[ \dimh \, \A_{-a,b} =2-\frac{2\lambda^2}{(a+b)^2}.\]
\end{rem}
	
	\begin{figure}[h!]
	\centering
		\includegraphics[width=0.3 \textwidth]{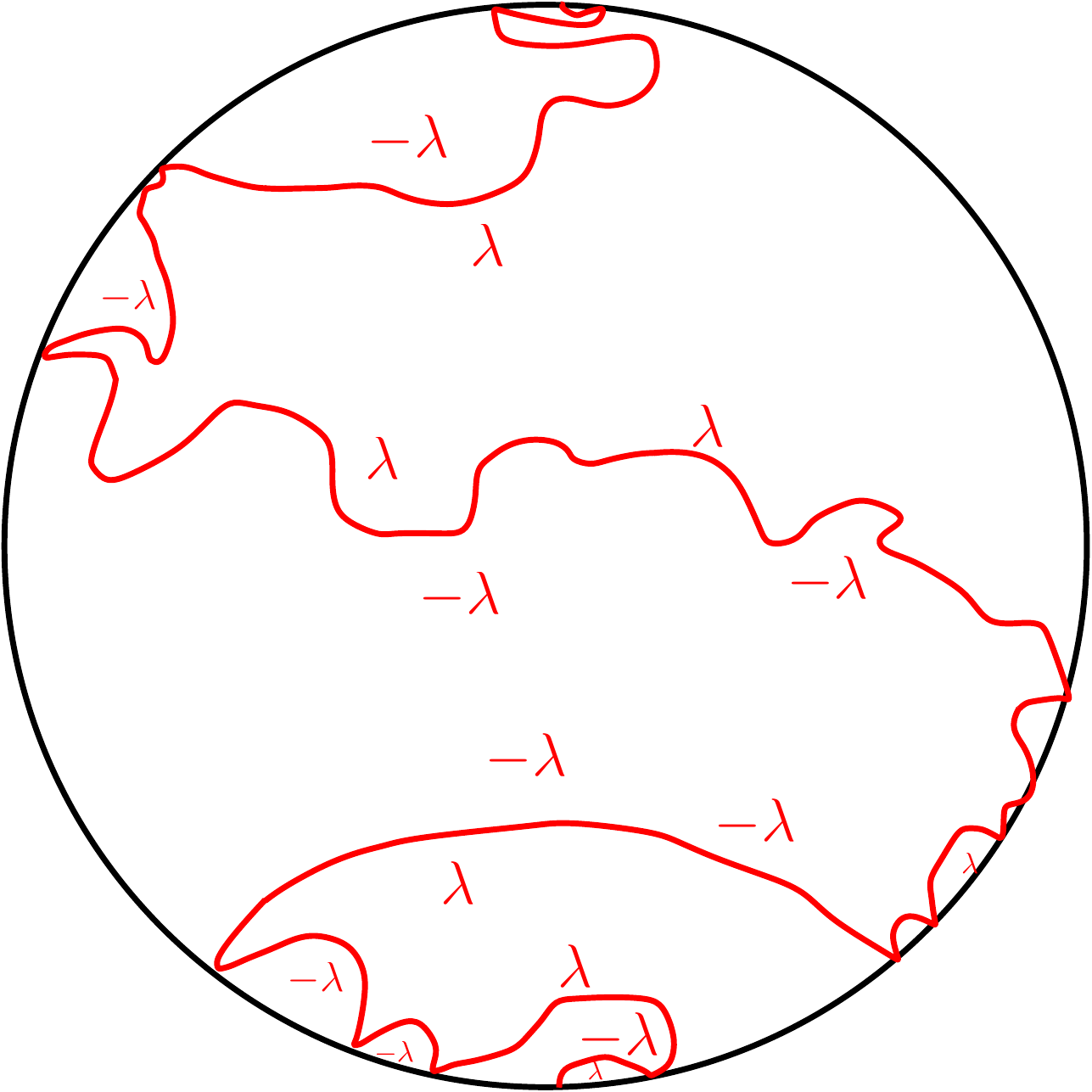}
		\hskip 0.5cm
		\includegraphics[width=0.3 \textwidth]{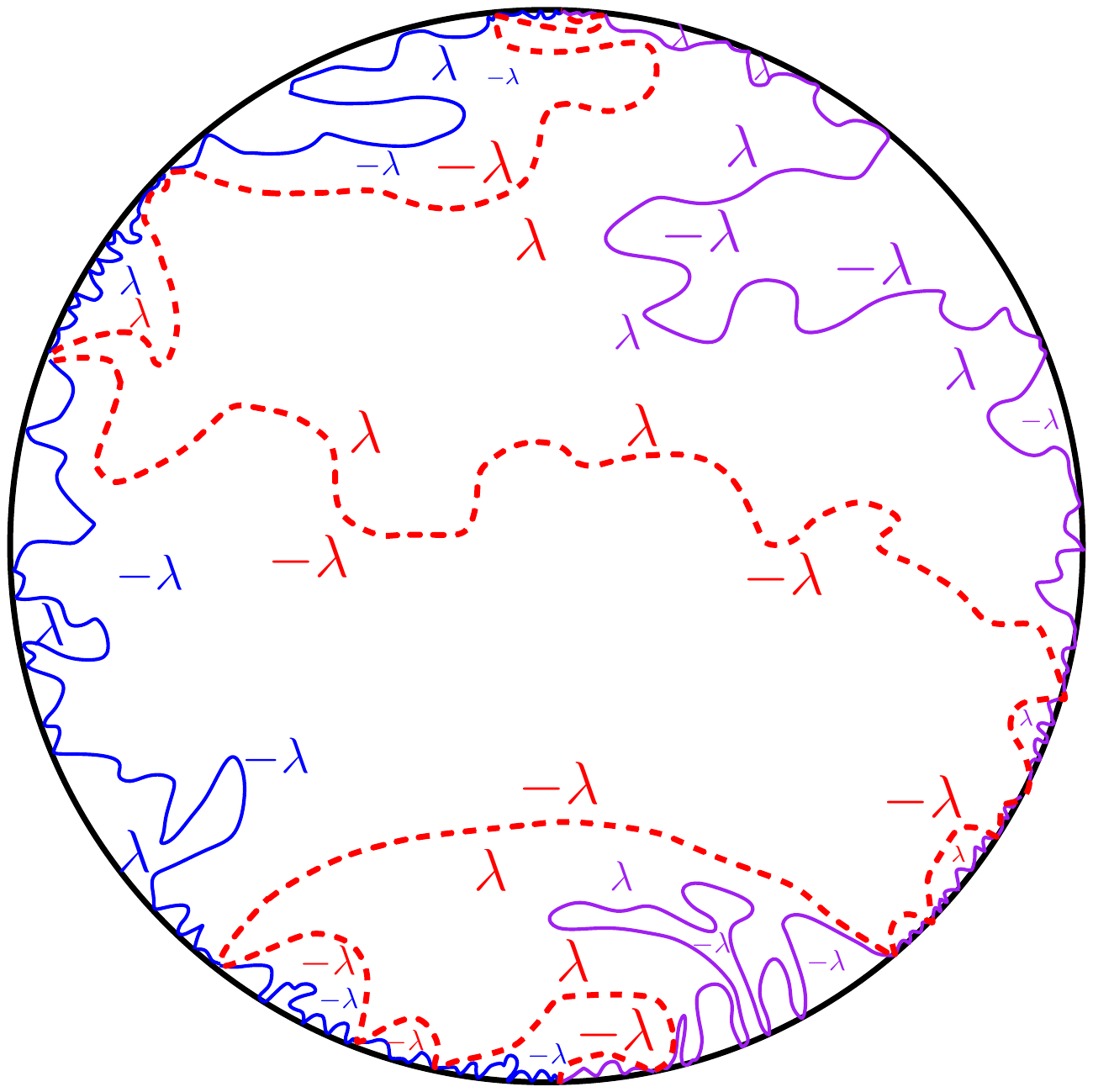}
		\caption{Construction of two-valued set.}
		\label{tvs-fig}
	\end{figure}

\subsection{Construction of two-valued sets}
For the convenience of the reader, we briefly recall here the construction of the two-valued sets for a zero-boundary GFF $\Gamma$ in $\D$. The construction in other domains is similar. To see that the local sets we construct are thin one can use Proposition~\ref{p.thin} together with an estimate on the expected dimensions of the sets, see Section 6 of \cite{ASW}.

We begin by noting that if $a$ or $b$ is $0$, then letting $A = \emptyset$ it is clear that $h_A = 0$ and by uniqueness (Proposition \ref{cledesc2}) we then have that $\A_{-a,b} = \emptyset$. 

\underline{ALE ($\A_{-\lambda,\lambda}$):}
We will now construct the set $\A_{-\lambda,\lambda}$ in $\D$. This we do in full detail, as this contains the main idea for the construction of every other two-valued set. Let $\eta$ be the (zero-height) level line from $-i$ to $i$ of $\Gamma$. Then $\eta$ is an $\textup{SLE}_4(-1,-1)$ curve which divides $\D$ into components $\Os_j^1$, $j \in J$ for some index set $J$, i.e., $\D \setminus \eta = \cup_{j \in J} \Os_j^1$. As the boundary values of the harmonic function $h_\eta$ are $-\lambda$ on the left of $\eta$ and $\lambda$ on the right  and $0$ on $\partial\D$, we have by the domain Markov property that inside each $\Os_j^1$ we have an independent GFF $\Gamma_j^1$ with boundary value $0$ on $\partial \Os_j^1 \cap \partial\D$ and on $\partial \Os_j^1 \cap \eta$ the boundary value is $-\lambda$ if $\Os_j^1$ is to the left and $\lambda$ if $\Os_j^1$ is to the right of $\eta$. Thus, we can begin iterating.

Assume that $\Os_j^1$ lies to the right of $\eta$ and let $w_j^1$ and $z_j^1$ be the start- and endpoints, respectively, of the clockwise arc $\partial \Os_j^1 \cap \partial\D$. Next, explore a level line $\eta_j^1$ of $\Gamma_j^1$ from $w_j^1$ to $z_j^1$. Again, on this curve, the boundary values of the harmonic function are $-\lambda$ and $\lambda$, but due to choosing $\eta_j^1$ to travel from $w_j^1$ to $z_j^1$, the side with boundary value $\lambda$ is the one closer to $\eta$. Thus, in the region enclosed by $\eta$ and $\eta_j^1$ (it is indeed only one region, as $\eta_j^1$ is an $\textup{SLE}_4(-1)$ curve attracted to $\partial\D$ and will hence not hit $\eta$), the harmonic function has constant boundary value $\lambda$, and is hence constant of value $\lambda$. The same is done in the domains to the left (but with $z_j^1$ as the staring point and $w_j^1$ as the endpoint of the clockwise arc of $\partial \Os_j^1 \cap \partial\D$), and then the harmonic function is $-\lambda$ in the region enclosed between the two curves. See Figure~\ref{tvs-fig}.

Doing this in every $\partial \Os_j^1$, and writing $A^1 = \eta \cup_j \eta_j^1$ we see that the harmonic function $h_{A^1}$ is constant in each bounded component of $\mathbb{C} \setminus A^1$. In the components of $\D \setminus A^1$ with an arc of $\partial\D$ as part of its boundary we again have boundary conditions $0$ on $\partial\D$ and $\pm \lambda$ on $\eta_j^1$ (now $-\lambda$ if the region is to the right of $\eta$ and $\lambda$ if it is on the left). Thus, we are in the same setting as in the first iteration and we can explore new level lines $\eta_j^2$ so that if $A^2 = A^1 \cup_j \eta_j^2$, then $h_{A^2}$ is constant on each bounded component of $\mathbb{C} \setminus A^2$. Thus, proceeding with this, we end up with a set $A$ such that $h_A \in \{ -\lambda,\lambda\}$, that is, $A = \A_{-\lambda,\lambda}$.

\underline{$a = -n_1 \lambda$ and $b = n_2 \lambda$, $n_1,n_2 \in \mathbb{N}\setminus \{0\}$:}
Pick a countable dense subset $S$ of $\D$. Choose one $z \in S$ and construct $\A_{-\lambda,\lambda}$. If $h_{\A_{-\lambda,\lambda}}(z) \in \{-n_1 \lambda, n_2 \lambda\}$, then we are done for this $z$. If not, then construct the two-valued set $\A_{-\lambda,\lambda}$, denote it by $\A_{-\lambda,\lambda}^2$, of the zero-boundary GFF in the component of $\D \setminus \A_{-\lambda,\lambda}$ containing $z$. Then $h_{\A_{-\lambda,\lambda} \cup \A_{-\lambda,\lambda}^2}(z) = h_{\A_{-\lambda,\lambda}}(z) + h_{\A_{-\lambda,\lambda}^2}(z)$ and if $h_{\A_{-\lambda,\lambda} \cup \A_{-\lambda,\lambda}^2}(z) \in \{ -n_1\lambda, n_2\lambda \}$, we stop. Otherwise, we continue the iteration until we reach that value. Doing this for every $z \in S$ gives the set $\A_{-a,b}$.

\underline{$a+b=2\lambda$:}
Let $c = (b-a)/2$. Then $c \in (-\lambda,\lambda)$ and repeating the exact same construction as for $\A_{-\lambda,\lambda}$, but with level lines of height $c$ (that is, level lines on which the harmonic function takes values $c-\lambda$, $c+\lambda$), we instead get components in which the harmonic function takes values $c+\lambda = b$ and $c-\lambda = -a$. 

\underline{$a+b = n\lambda$, $3 \leq n \in \mathbb{N}$:}
Let $c \in (-\lambda,\lambda)$ be such that there exists non-negative integers $n_1$ and $n_2$ such that $-a = c-n_1 \lambda$ and $b = c + n_2 \lambda$ (i.e., such that $n_1 + n_2 = n$). Start with $\A_{c-\lambda,c+\lambda}$ and in the components where the harmonic function has the value $c-\lambda$ construct $\A_{-(n_1-1)\lambda,(n_2+1)\lambda}$ and in the components where the harmonic functions has value $c+\lambda$ construct $\A_{-(n_1+1)\lambda,(n_2-1)\lambda}$.

\underline{$a+b>2\lambda$:}
Assume by symmetry that $a > \lambda$. Let $c \in [0,\lambda)$ and $n_1,n_2 \in \mathbb{N}$ be such that $b = c +n_1 \lambda$ and $b-n_2 \lambda \in [-a,-a+\lambda)$ and write $d = a+b-n_2 \lambda$. We start with an $\A_{b-n_2\lambda,b}$ (which is possible because necessarily, $b-n_2\lambda<0$ and $n_2 \geq 2$). In the connected components where the harmonic function is $b$ we are done, and hence stop, but in the components where it takes the value $b-n_2\lambda$, iterate $\A_{-d,-d+n_2\lambda}$. Then we get connected components where the harmonic function takes the values $-a$ (there we stop) and $b-d$. Continue by iterating $\A_{d-n_2\lambda,d}$ in the $(b-d)$-components to again get components where the harmonic function takes values $b$ and $b-n_2\lambda$. Continue an alternating iteration of $\A_{-d,-d+n_2\lambda}$ and $\A_{d-n_2\lambda,d}$ in the components where the values $-a$ and $b$ are not taken, and finally take the closure of the union of all of the constructed sets to get $\A_{-a,b}$.

	\subsection{Imaginary multiplicative chaos}
	We now recall the construction and main results related to  the purely imaginary chaos.
	 Let $\Gamma$ be a Dirichlet boundary condition GFF in $\D$ (extended by $0$ to $\D^c$) and let $\sigma \in \R$. For $\epsilon > 0$, we will write $\Gamma_\epsilon(z) = \Gamma(\rho_\epsilon(z))$ for the circle averages obtained by taking $\rho_\epsilon(z)$ to be the uniform probability measure on the circle of radius $\epsilon$ about $z$. If $D$ is a domain, we write $d(z,\partial D)$ for the distance from $z$ to $\partial D$ and recall that $r_D(z)$ denotes the conformal radius of $D$ seen from $z$. Then if $d(z,\partial \D) > \epsilon$, for $\epsilon$ fixed, $z \mapsto \Gamma_\epsilon(z)$ is a random H\"older continuous function. If $z,w$ satisfy $|z-w| > 2\eps $ and are at distance greater than $\eps$ from $\partial \D$, then since the Green's function is harmonic in both variables we have
	 \[
	 \E[\Gamma_\eps(z) \Gamma_\eps(w)] = G_\D(z,w).
	 \]
	 Moreover, $\E[\Gamma_\eps(z)^2] = \log \eps^{-1} + \log r_{\D}(z)$.
	For $\eps>0$, define
	\[
	\mathcal{V}_\epsilon^{i\sigma}(z) := e^{i\sigma \Gamma_{\epsilon}(z)-\frac{\sigma^2}{2}\ln \epsilon}
	\]
	and note that
	\[
	\E[ \mathcal{V}_\epsilon^{i\sigma}(z)] = r_{\D}(z)^{-\frac{\sigma^2}{2}}.
	\]
	Proposition~3.1 of \cite{JSW} gives convergence of a slightly differently normalized version of $\mathcal{V}_\epsilon^{i\sigma}(z)$ (see also \cite{LRV}): if $0 < |\sigma| <\sqrt{2}$, then as $\epsilon \to 0$ the random functions $\widetilde{\mathcal{V}}_\epsilon^{i\sigma}(z):= r_{\D}(z)^{\frac{\sigma^2}{2}}\mathcal{V}_\epsilon^{i\sigma}(z)$ (extended by $0$ to $\D^c$) converge in probability in the Sobolev space $H^{s}(\mathbb{C})$, $s < -1$, to a non-trivial random element $\widetilde{\mathcal {V}}^{i \sigma}$ supported on $\overline{\D}$.\footnote{This result in fact holds for any simply connected and bounded domain $D$, but we only need to consider the unit disc.} We now define 
\[
\mathcal{V}^{i \sigma}=r_{\D}(z)^{-\frac{\sigma^2}{2}}\widetilde{\mathcal{V}}^{i \sigma}.
\]
	Then for a real-valued test function $f \in C_0^\infty(\D)$\footnote{We will always consider real-valued test functions.}, $(\mathcal{V}^{i \sigma}, f)$ is well-defined and if we take  $(\mathcal{V}_\epsilon^{i\sigma},f)= \int_{\D} \mathcal{V}_\epsilon^{i\sigma}f\, dz$, then $(\mathcal{V}_\eps^{i \sigma}, f) \to (\mathcal{V}^{i \sigma}, f)$ in $L^2(\P)$ and 
\begin{align}\label{e. one point}
\E\left[ (\mathcal{V}^{i\sigma}, f) \right]= \int_\D f(z) r_\D(z)^{-\frac{\sigma^2}{2}} dz. 
\end{align}
In fact, assuming $f$ is bounded, measurable, and with support compactly contained in $\D$, convergence of $(\mathcal{V}_\epsilon^{i\sigma},f)$ occurs in all $L^p(\P)$ spaces, $p \ge 1$, and we interpret $(\mathcal{V}^{i\sigma},f)$ as this limiting random variable for which \eqref{e. one point} holds as well. See e.g. the proof of Corollary~3.11 of \cite{JSW}. See Proposition~\ref{p. n point function} below for formulas for the $n$-point correlations.

We will also consider the real part of $\mathcal{V}^{i\sigma}$, which we think of as the cosine of the field $\sigma \Gamma$. We have the following definition, also appearing in \cite{JSW}. Suppose $|\sigma| < \sqrt{2}$. Then for any function bounded and measurable function $f$, with support compactly contained in $\D$, we define
\[
(\cos(\sigma \Gamma), f) = \lim_{\epsilon \to 0+} \int_\D \epsilon^{-\frac{\sigma^2}{2}} \cos(\sigma \Gamma_\epsilon(z)) f(z) dz =  \textup{Re} (\mathcal{V}^{i\sigma},f),
\]
where the convergence again takes place in $L^p(\P)$ for any $p \ge 1$. 

Next, we need to be able to compute correlations of the imaginary chaos. We refer to \cite{JSW} for additional discussion. Suppose $D$ is a simply connected domain. Let $m,n\in \N$, ${\bf z}=(z_1,\dots,z_{m+n})= (x_1,\dots,x_m,y_1,\dots,y_n)$, and define the $m+n$ point correlation function as follows
\begin{align}\label{e. n correlations}
& \bigg<\prod_{j=1}^m \mathcal{V}^{i \sigma}(x_j)  \prod_{k=1}^n \mathcal{V}^{-i \sigma}(y_k)\bigg>_{D} := \nonumber
 \\
& \left (\prod_{l=1}^{m+n} r_D(z_l)^{-\frac{\sigma^2}{2}} \right) \frac{\exp\left (-\sigma^2\left(  \sum_{j<j'} G_D(x_j,x_{j'}) + \sum_{k<k'} G_D(y_k,y_{k'}) \right)\right )}{\exp\left (-\sigma^2\left(  \sum_{j,k} G_D(x_j,y_k) \right)\right) }
\end{align}
Note that the correlation function is symmetric with respect to exchanging the $x$ and $y$ vectors.
To give a little bit more intuition for the correlation function, let us recall that $G_D(z,w)=-\log(|z-w|)+u(z,w)$ where $u(z,w)$ solves the Dirichlet problem with boundary data $\log|\cdot - w|$. So when $D$ is bounded and $w$ is in a compact subset $K \subset D$, $u$ is also bounded. It follows that
\begin{align*}
\bigg<\prod_{j=1}^m \mathcal{V}^{i \sigma}(x_j)  \prod_{k=1}^n \mathcal{V}^{-i \sigma}(y_k)\bigg>_{D}  \asymp_K \frac{\left( \prod_{j<j'} |x_j-x_{j'}|^{\sigma^2}\right)\left( \prod_{k<k'} |y_k-y_{k'}|^{\sigma^2}\right)}{\left( \prod_{j,k} |y_j-x_{k}|^{\sigma^2}\right)}
\end{align*}
where the implicit constant depends on the fixed compact $K\subset D^{m+n}$ in which ${\bf z}$ varies. We can now write down the correlation function of $\mathcal{V}^{i\sigma}$. More precisely, by Gaussian calculations and Proposition 3.6(ii), Lemma 3.10 and Corollary 3.11 of \cite{JSW} giving sufficient integrability, we have the next proposition. While the first and last results are stated for mollifications of the field, rather than the approximation by circle averages, the circle averages $\Gamma_{\epsilon_n}$ constitute a standard approximation of the field $\Gamma$ (see Definition 2.7 of \cite{JSW}) for any sequence $\epsilon_n \searrow 0$, and using these properties, the proof of Proposition 3.6(ii) and hence Corollary 3.11 of \cite{JSW} for circle averages are the same as for mollifications.

\begin{prop}\label{p. n point function}
Let $\mathcal{V}^{i \sigma}$ be the imaginary chaos in $\D$ with  $0 < \sigma < \sqrt{2}$.
For all measurable, bounded functions $(f_j)_{j=1}^m$, $(g_k)_{k=1}^n$ with supports compactly contained in $\D$,
\begin{align*}
	&\E\left[\Bigg(\prod_j (\mathcal{V}^{i \sigma},f_j) \Bigg)\Bigg(\prod_k \overline{(\mathcal{V}^{i \sigma},g_k)} \Bigg)  \right] \\
	&= \int_{\D^{n+m}} \bigg<\prod_{j=1}^m \mathcal{V}^{i \sigma}(x_j)  \prod_{k=1}^n \mathcal{V}^{-i \sigma}(y_k)\bigg>_{\D} \prod_{j,k} f_j(x_j) g_k(y_k) dx_j dy_k.
\end{align*}
\end{prop}



\begin{rem}
An important difference for the imaginary chaos compared to the ``standard'' case of real Gaussian multiplicative chaos, the limiting measure obtained by passing to the limit with $\eps^{\alpha^2/2}e^{\alpha \Gamma_\eps} dz$. In the complex case, when there is convergence in $L^2$ there is also convergence for all $L^p$, and when there is no convergence in $L^2$, then the field does not converge in $L^1$. It turns out that the moments determine the distribution, see Theorem 1.3 in \cite{JSW}.
\end{rem}
%

\begin{rem}
Similar results hold for imaginary chaos for given $\sigma$ defined on bounded simply connected domains $D$ satisfying the condition
 \begin{equation}\label{e.condition}
 	\int_D d(z,\partial D)^{-\frac{\sigma^2}{2}}dz<\infty.
 \end{equation}
 Note that for the unit disc, this condition is satisfied for all  $\sigma < \sqrt{2}$.
Moreover, if we know that the Minkowski dimension of the boundary is strictly smaller than $2$ (as is the case of SLE$_4$-type loops which have dimension $3/2$), then \eqref{e.condition} is satisfied for small enough $\sigma$. 

However, without this information \eqref{e.condition} may fail within the class of Holder domains even for small $\sigma$. Compare this with the case of SLE$_\kappa, \kappa \in (4,8)$ for which the expected $(1+\kappa/8)$-dimensional Minkowski content away from the start and end points is finite for all bounded simply connected domains.
\end{rem}

\section{Imaginary chaos conditioned on a two-valued set}

\subsection{Main estimate}
Suppose we are given the two-valued set $\A_{-a,a}$. The components of $\D \setminus \A_{-a,a}$ is a collection of simply connected domains each with a well-defined Green's function. For $z \in \D \setminus \Aa$ we let $\Os(z)$ denote the connected component of $\D \setminus \Aa$ containing $z$. It is convenient to think of this collection of Green's functions as one function and we shall make the following definition.
\begin{align*}
		G_{\D \setminus \A_{-a,a}}(z,w) = \begin{cases}
		G_\Os(z,w), &\text{if } \Os(z) = \Os(w) = \Os,\\
		0, &\text{otherwise}. 
		\end{cases}
\end{align*}
Moreover, we write $r_{\D \setminus \Aa}(z) = r_{\Os(z)}(z)$ and we define 
\begin{align*}
    \bigg<\prod_{j=1}^m \mathcal{V}^{i \sigma}(x_j)  \prod_{k=1}^n \mathcal{V}^{-i \sigma}(y_k)\bigg>_{\D\setminus\Aa}
\end{align*}
by \eqref{e. n correlations}, replacing $G_D$ and $r_D$ by $G_{\D \setminus \Aa}$ and $r_{\D \setminus \Aa}$, respectively.
	
The objective of the following section is to prove the following proposition.
\begin{prop}\label{p. general conditioning}
Let $\Gamma$ be a GFF in $\D$, let $a \geq \lambda$, $0\leq\sigma<\sigma_c:= \lambda/a \le 1$ and suppose $\mathcal{V}^{i \sigma}$ is the imaginary chaos. Then for any set of measurable, bounded functions $(f_j)_{j=1}^m, (g_k)_{k=1}^n$ with supports compactly contained in $\D$,
\begin{align}\label{oct1.1}
	&\E\left[\Bigg(\prod_j (\mathcal{V}^{i \sigma},f_j) \Bigg) \Bigg(\prod_k \overline{(\mathcal{V}^{i \sigma},g_k)} \Bigg)\Bigg| \, \F_{\A_{-a,a}} \right]\\
	&\quad \quad= \int_{(\D\setminus \A_{-a,a})^{n+m}} e^{i\sigma \left( \sum_{j} h_{\A_{-a,a}}(x_j)-\sum_k h_{\A_{-a,a}}(y_k)\right) } \nonumber\\
	&\quad \quad \quad \quad \times \bigg<\prod_{j=1}^m \mathcal{V}^{i \sigma}(x_j)  \prod_{j=1}^n \mathcal{V}^{-i \sigma}(y_j)\bigg>_{\D \setminus \Aa}  \prod_{j,k} f_j(x_j) g_k(y_k) dx_j dy_k, \nonumber
\end{align}
where $\F_{\A_{-a,a}}$ is the sigma algebra generated by $(\A_{-a,a},\Gamma_{\A_{-a,a}})$ as in Definition \ref{defn:sigma algebra}.
\end{prop}
This result is in a sense simply a modification of the main result of \cite{APS}. The principal difference is the fact that instead of conditioning only one term of the product we can actually work with many of them due to the good integrability properties. 
			
For the attentive reader, it may come as a surprise the fact that we ask $\sigma<\sigma_c\leq 1$, as we would like to have the result for all $\sigma<\sqrt 2$. The fact that this result actually does not hold for the whole range of $\sigma$ is in some sense what makes combining imaginary chaos with two-valued sets interesting.

\subsection{One-point function conditioned on a two-valued set}

Let us now describe how $\mathcal{V}^{i \sigma}$ looks when one conditions on $\A_{-a,a}$.
	
\begin{lemma}\label{l. one point}
Suppose $a\geq \lambda$ and $\sigma < \sigma_c = \lambda/a$. Then, for each bounded measurable function $f$,
\begin{equation} \label{e. conditional expectation}
	\E\left[(\mathcal{V}^{i \sigma}, f)| \F_{\A_{-a,a}} \right] = \int_\D f(z) r_{\D \setminus \A_{-a,a}}(z)^{-\sigma^2/2} e^{i\sigma h_{\A_{-a,a}}(z)} dz.
\end{equation}
\end{lemma}
	
\begin{proof}
As $(\mathcal{V}_\epsilon^{i\sigma}, f)$ converges in $L^1(\P)$ to $(\mathcal{V}^{i \sigma}, f)$ it follows that as $\epsilon \to 0$, $\E[(\mathcal{V}_\epsilon^{i\sigma}, f) \mid \F_{\A_{-a,a}}]$ converges to $\E[(\mathcal{V}^{i\sigma}, f) \mid \F_{\A_{-a,a}}]$ in $L^1(\P)$. Hence, it suffices to show that $\E[(\mathcal{V}_\epsilon^{i\sigma}, f) \mid \F_{\A_{-a,a}}]$ converges in probability to the right-hand side of \eqref{e. conditional expectation}. Since $\A_{-a,a}$ is a thin local set for $\Gamma$, given $\F_{\A_{-a,a}}$ we have the decomposition $\Gamma = h_{\A_{-a,a}}+\Gamma^{\A_{-a,a}}$. This implies that the conditional expectation of $(\mathcal{V}_\epsilon^{i\sigma}, f)$ given $\F_{\A_{-a,a}}$ can be written		
\begin{linenomath}
	\begin{align*}
		&\int_{d(z,\A_{-a,a}) > \epsilon} f(z)e^{i\sigma h_{\A_{-a,a}}(z)}\E\left[e^{i\sigma \Gamma_\epsilon(z) - \frac{\sigma^2}{2}\ln \epsilon} \Big| \, \F_{\A_{-a,a}} \right]dz \\
		& \qquad + \int_{d(z,\A_{-a,a}) \leq \epsilon}\E\left[f(z) e^{i\sigma \Gamma_\epsilon(z) - \frac{\sigma^2}{2}\ln \epsilon}\Big| \, \F_{\A_{-a,a}} \right]dz\\
		&= \int_{d(z,\A_{-a,a}) > \epsilon} f(z) r_{\D \setminus \A_{-a,a}}(z)^{-\frac{\sigma^2}{2}} e^{i\sigma h_{\A_{-a,a}}(z)} dz \\
		& \qquad + \int_{d(z,\A_{-a,a}) \leq \epsilon}\E\left[f(z) e^{i\sigma \Gamma_\epsilon(z) - \frac{\sigma^2}{2}\ln \epsilon}\Big| \, \F_{\A_{-a,a}} \right]dz.
	\end{align*}
\end{linenomath}
By dominated convergence, the first term in the last expression converges to the right-hand side of \eqref{e. conditional expectation} almost surely so the result follows if we show that the second term converges to $0$ in probability. This fact is the content of the next lemma, and assuming that lemma the proof of this one is complete.  \end{proof}

\begin{lemma}\label{c. boundary to 0}			
Suppose $a\geq \lambda$ and $\sigma<\sigma_c=\lambda/a$. Then, for each bounded measurable function $f$,
\begin{align*}
    \int_\D \E\left[f(z)e^{i\sigma \Gamma_\epsilon(z)-\frac{\sigma^2}{2}\ln \epsilon} \middle| \Fa \right] \1_{\{d(z,\A_{-a,a})<\epsilon\}} dz
\end{align*}
converges to $0$ in $L^{1}(\P)$, and thus, in probability, as $\epsilon \rightarrow 0$.
\end{lemma}
\begin{proof}
We begin by noting that  
\begin{align*}
    &\left| \int_{d(z,\Aa) \leq \epsilon} \E \left[ f(z) e^{i\sigma \Gamma_\epsilon (z) - \frac{\sigma^2}{2} \ln \epsilon} \Big| \Fa \right] dz \right| \\
    &\leq \int_\D \1_{\{d(z,\Aa) \leq \epsilon\}}(z) |f(z)| \epsilon^{-\frac{\sigma^2}{2}} dz.
\end{align*}
Thus, taking expectations we have
\begin{align*}
    &\E\left[ \left| \int_{d(z,\Aa) \leq \epsilon} \E \left[ f(z) e^{i\sigma \Gamma_\epsilon (z) - \frac{\sigma^2}{2} \ln \epsilon} \Big| \Fa \right] dz \right| \right] \\
    &\leq \E\left[ \int_\D \1_{\{d(z,\Aa)\leq \epsilon\}}(z) |f(z)| \epsilon^{-\frac{\sigma^2}{2}} dz \right] \\
    &\leq \| f \|_{L^\infty} \epsilon^{-\frac{\sigma^2}{2}} \int_\D \P(d(z,\Aa) \leq \epsilon) dz  \\
    &\leq \| f \|_{L^\infty} \epsilon^{-\frac{\sigma^2}{2}} \int_\D \P(r_{\D \setminus \Aa}(z) \leq 4\epsilon) dz,
\end{align*}
since $d(z,\Aa)\leq\epsilon$ implies that $r_{\D\setminus\Aa}(z) \leq 4\epsilon$. By Lemma \ref{l.1point}, this is bounded by a constant times
\begin{align*}
    &\epsilon^{\frac{\sigma_c^2-\sigma^2}{2}} \int_{(1-8\epsilon)\D} r_\D(z)^{-\frac{\sigma^2}{2}} dz + \epsilon^{-\frac{\sigma^2}{2}} \int_{(1-8\epsilon) < |z| < 1} \P(r_{\D \setminus \Aa}(z) \leq 4\epsilon) dz \\
    &= O(\epsilon^{\frac{\sigma_c^2-\sigma^2}{2}})  + O(\epsilon^{1-\frac{\sigma^2}{2}}),
\end{align*}
and since $\sigma < \sigma_c \leq 1$ we are done.
\end{proof}

\subsection{The general case: Proof of Proposition \ref{p. general conditioning}}
The proof of the general case is similar to the one point estimate, with some additional technical complications. The main idea is the same: we show that the points near the two-valued set do not contribute to the integral but we also need to handle terms with points on the `diagonal'. 

\begin{proof}[Proof of Proposition \ref{p. general conditioning}]
We want to pass to the limit as $\eps \to 0$ with the following expression.
\begin{equation}\label{e.n-point epsilon}
	\E\left[ \Bigg( \prod_j \int_{\D} \mathcal{V}_\epsilon^{i\sigma}(x_j) f_j(x_j) dx_j\Bigg)\Bigg ( \prod_k \int_{\D} \overline{\mathcal{V}_\epsilon^{i\sigma}(y_k)} g_k(y_k) dy_k\Bigg)\Bigg| \F_{\A_{-a,a}}\right].
\end{equation}
As in the proof of Lemma \ref{l. one point}, we have that  
\[
\Bigg(\prod_j (\mathcal{V}_\epsilon^{i\sigma},f_j) \Bigg)\Bigg(\prod_k \overline{(\mathcal{V}_\epsilon^{i\sigma},g_k)} \Bigg) \to \Bigg(\prod_j (\mathcal{V}^{i \sigma},f_j) \Bigg)\Bigg(\prod_k \overline{(\mathcal{V}^{i \sigma},g_k)} \Bigg)
\]
in $L^1(\P)$, so the random variables \eqref{e.n-point epsilon} converge in $L^1(\P)$. We will show that we have convergence in probability to the right-hand side of \eqref{oct1.1}. 
For this, we first define the small sets 
\[
A_\epsilon=\{z\in \D: d(z,\A_{-a,a}) \leq \epsilon\},
\]
and with $N=m+n$,
\[
B_\eps = \{(z_1, \dots, z_N) \in \D^N: \exists j \neq k \textup{ such that } |z_j-z_k| \le 2\eps\}.
\]
We split each integral appearing in \eqref{e.n-point epsilon} in a `good' part corresponding to integrating over \[G_\eps= (\D^N \setminus A_\eps^N) \cap (\D^N \setminus B_\eps) \subset \D^N,\] and a `bad' part corresponding to integrating over $\D^N \setminus G_\eps$. We will prove below that the bad part converges to $0$ in probability. Assuming this, it follows that the good part also converges in probability (since it is equal to a difference of two random variables both converging in probability) and we want to show that the convergence is towards the right-hand side of \eqref{oct1.1}. To see this, note that if $z,w \in \D \setminus A_\eps$ and $|z-w| > 2\eps$ then $\E[\Gamma_\eps(z) \Gamma_\eps(w) \mid \F_{\A_{-a,a}}] = G_\Os(z,w)$ if $z,w$ are in the same component $\Os$ and $0$ otherwise. So we get the following formula:
\begin{align*}
	&\E\left[ \int_{G_\eps} \Bigg( \prod_j \mathcal{V}_\epsilon^{i\sigma}(x_j) f_j(x_j) dx_j\Bigg)\Bigg( \prod_k \overline{\mathcal{V}_\epsilon^{i\sigma}(y_k)} g_k(y_k) dy_k\Bigg)\Bigg| \F_{\A_{-a,a}}\right]\\
	&= \int_{G_\eps} e^{i\sigma \left( \sum_{j} h_{\A_{-a,a}}(x_j)-\sum_k h_{\A_{-a,a}}(y_k)\right) }  \\
	&\qquad \qquad \times \bigg<\prod_{j=1}^m \mathcal{V}^{i \sigma}(x_j)  \prod_{k=1}^n \mathcal{V}^{-i \sigma}(y_k)\bigg>_{\D \setminus \Aa}  \prod_{j,k} f_j(x_j) g_k(y_k) dx_j dy_k.
\end{align*}
Since $\Aa$ is a two-valued set and $N < \infty$, $\sum_{j} h_{\A_{-a,a}}(x_j)-\sum_k h_{\A_{-a,a}}(y_k)$ can take only a finite number of values when $x_j, y_k \in \D \setminus \Aa$. Moreover, $G_\eps$ is clearly increasing to $(\D \setminus \Aa)^N$ (up to a null-set) as $\eps \to 0$ and so by considering separately the positive and negative parts of the product of test functions, we can apply the monotone convergence theorem (along a subsequence) to see that this term indeed converges to the right-hand side of \eqref{oct1.1} in probability.

Now we turn to proving that the integral over $\D^N \setminus G_\epsilon$ converges to $0$ in probability. We begin by considering the set 
\begin{align*}
    (\D^N \setminus A_\epsilon^N)^c = \{(z_1,\dots,z_N): d(z_j,\Aa) \leq \epsilon \textup{ for some } j \in \{1,\dots,N\}\}.
\end{align*}
To make notation simpler, we write
\begin{align*}
    \hat{\V}_\epsilon^{i\sigma}(z_l) = \begin{cases} \V_\epsilon^{i\sigma}(z_l) &\textup{if } l \leq m, \\
                                        \overline{\V_\epsilon^{i\sigma}(z_l)} &\textup{if } l > m,\end{cases}
\end{align*}
and
\begin{align*}
    \hat{\varphi}_l(z_l) = \begin{cases} f_l(z_l) &\textup{if } l \leq m, \\
                            g_l(z_l) &\textup{if } l > m.\end{cases}
\end{align*}
The integral over the set $(\D^N \setminus A_\epsilon^N)^c$ is a sum over terms of the form
\begin{align}\label{eq:genformA}
    &\E\left[ \Bigg( \prod_{{\alpha(l)} \leq L} \int_{A_\epsilon} \hat{\V}_\epsilon^{i\sigma}(z_{\alpha(l)}) \hat{\varphi}_{\alpha(l)}(z_{\alpha(l)}) dz_{\alpha(l)}\Bigg) \right.\\
    &\quad \left. \ \times\Bigg( \prod_{{\alpha(l)}> L} \int_{ \D\setminus A_\epsilon} \hat{\V}_\epsilon^{i\sigma}(z_{\alpha(l)}) \hat{\varphi}_{\alpha(l)}(z_{\alpha(l)}) dz_{\alpha(l)}\Bigg) \Bigg| \F_{\A_{-a,a}}\right], \nonumber
\end{align}
where $\alpha$ is a permutation of $\{1,\dots,N\}$ and $L \geq 1$. The absolute value of \eqref{eq:genformA} is bounded by
\begin{align}\label{eq:intcondexp}
    &\Bigg( \prod_{\alpha(l) \leq L} \int_{A_\epsilon} \epsilon^{-\frac{\sigma^2}{2}}|\hat{\varphi}_{\alpha(l)}(z_{\alpha(l)})| dz_{\alpha(l)} \Bigg) \nonumber\\
	&\times\E\left[ \Bigg| \prod_{\alpha(l)> L}\int_{ \D\setminus A_\epsilon} \hat{\V}_\epsilon^{i\sigma}(z_{\alpha(l)}) \hat{\varphi}_{\alpha(l)}(z_{\alpha(l)}) dz_{\alpha(l)}\Bigg| \Bigg| \F_{\A_{-a,a}}\right].
\end{align}
By the proof of Lemma \ref{c. boundary to 0}, the integrals over the $A_\epsilon$-sets converge in probability to $0$. Thus, we need only show that the conditional expectation of the integrals over $\D \setminus A_\epsilon$ is bounded as $\epsilon \rightarrow 0$. This requires an argument since what know at this point is that the integrals over $\D$ are bounded and that the integrals over $A_\epsilon$ converge to $0$ as $\epsilon \rightarrow 0$. But we can write the integrals over $\D \setminus A_\epsilon$ as differences of integrals over $\D$ and $A_\epsilon$ and use the triangle inequality. That is,
\begin{align*}
    &\left| \int_{\D \setminus A_\epsilon} \hat{\V}_\epsilon^{i\sigma}(z_{\alpha(l)}) \hat{\varphi}_{\alpha(l)}(z_{\alpha(l)}) dz_{\alpha(l)} \right| \\
    &\leq \left| \int_{\D} \hat{\V}_\epsilon^{i\sigma}(z_{\alpha(l)}) \hat{\varphi}_{\alpha(l)}(z_{\alpha(l)}) dz_{\alpha(l)} \right| +  \int_{A_\epsilon} \epsilon^{-\frac{\sigma^2}{2}} |\hat{\varphi}_{\alpha(l)}(z_{\alpha(l)})| dz_{\alpha(l)} .
\end{align*}
Thus, using this, \eqref{eq:intcondexp} is bounded by a sum of terms on the form
\begin{align}\label{eq:finalintA}
    &\Bigg( \prod_{\tilde{\alpha}(l) <\tilde{L}} \int_{A_\epsilon} \epsilon^{-\frac{\sigma^2}{2}}|\hat{\varphi}_{\tilde{\alpha}(l)}(z_{\tilde{\alpha}(l)})| dz_{\tilde{\alpha}(l)} \Bigg) \nonumber\\
	&\times\E\left[ \Bigg| \prod_{\tilde{\alpha}(l) \geq \tilde{L}}\int_{\D} \hat{\V}_\epsilon^{i\sigma}(z_{\tilde{\alpha}(l)}) \hat{\varphi}_{\tilde{\alpha}(l)}(z_{\tilde{\alpha}(l)}) dz_{\tilde{\alpha}(l)}\Bigg|  \Bigg| \F_{\A_{-a,a}}\right],
\end{align}
where $\tilde{L} > L$, and $\tilde{\alpha} = \tau \circ \alpha$, where $\tau$ is a permutation of $\{1,\dots,N\}$, fixing $\{1,\dots,L\}$. This term converges to $0$ in probability, since the integrals over $A_\epsilon$ converge to $0$ in probability as $\epsilon \rightarrow 0$ and since
\begin{align*}
    \Bigg| \prod_{\tilde{\alpha}(l)\geq \tilde{L}}\int_{\D} \hat{\V}_\epsilon^{i\sigma}(z_{\tilde{\alpha}(l)}) \hat{\varphi}_{\tilde{\alpha}(l)}(z_{\tilde{\alpha}(l)}) dz_{\tilde{\alpha}(l)}\Bigg|
\end{align*}
converges in $L^1(\P)$.

Thus, we have shown the convergence of the conditional expectation of the integrals over the sets $G_\epsilon$ and $(\D \setminus A_\epsilon^N)^c$, and to finish the proof it remains to show that the conditional expectation of the integrals over the set $B_\epsilon \setminus A_\epsilon^N$ converges to $0$ in probability.


Recall that $B_\epsilon$ consists of sets where some points are within $2\epsilon$ distance of each other and the rest being of distance greater than $2\epsilon$ from every other point. That is, $B_\epsilon$ is a union of sets on the form
\begin{align*}
    B_\epsilon^{(z_{p(1)},z_{q(1)}),\dots,(z_{p(k)},z_{q(k)})}= \, &\{(z_1,\dots,z_N) \in \D^N: |z_{p(i)}-z_{q(i)}| \leq 2\epsilon, \textup{ for } i \leq k, \\
    &\quad \textup{and } |z_{p(i)}-z_{q(i)}| > 2\epsilon \textup{ for } i > k \}
\end{align*}
where $p,q:\{1,\dots,N^2-N\} \rightarrow \{1,\dots,N\}$ are such that for each $i$, $p(i) < q(i)$ and for each $j<l$, there is an $i$ such that $(j,l) = (p(i),q(i))$. In this form it is somewhat hard to evaluate the integral $B_\epsilon$, as the different coordinates depend on each other and can not be separated as cylinder sets. For this reason we rewrite it as follows. Define, for $p,q$ as above, the indicator functions
\begin{align*}
    U_\epsilon^{(z_{p(1)},z_{q(1)}),\dots,(z_{p(k)},z_{q(k)})} = \1_{\{|z_{p(1)}-z_{q(1)}| \leq 2\epsilon \}} \cdots \1_{\{|z_{p(k)}-z_{q(k)}|\leq 2\epsilon \}}.
\end{align*}
Then the indicator function, $\mathscr{I}_\epsilon^{p,q}(k)$, of $B_\epsilon^{(z_{p(1)},z_{q(1)}),\dots,(z_{p(k)},z_{q(k)})}$ can be written as a linear combination of different $U_\epsilon^{(\cdot,\cdot),\dots,(\cdot,\cdot)}$. More precisely, by the principle of inclusion-exclusion,
\begin{align}\label{eq:inclexcl}
    \mathscr{I}_\epsilon^{p,q}(k) = \sum_{\substack{i_1,\dots,i_j \\ j}} (-1)^j U_\epsilon^{(z_{p(1)},z_{q(1)}),\dots,(z_{p(k)},z_{q(k)}),(z_{p(i_1)},z_{q(i_1)}),\dots,(z_{p(i_j)},z_{q(i_j)})}
\end{align}
where the sum is over subsets $\{i_1,\dots,i_j\} \subseteq \{k+1,\dots,N(N-1)\}$, $i_1 < \dots < i_j$, $0\leq j \leq N(N-1)-k$.
This is very useful, because when expressed in terms of these indicators, we get a sum of integrals in which we can separate the points that lie close to each other (the points in the sets $U_\epsilon^{(\cdot,\cdot),\dots,(\cdot,\cdot)}$) from the rest, and write each integral as a product of several integrals. It follows that the conditional expectation of the integral over $B_\epsilon \setminus A_\epsilon^N$ can be written as a sum of terms on the form
\begin{align*}
    &\E\Bigg[ \prod_{\substack{1 \leq l \leq N \\ l \notin R_{p,q}^{k}(i_1,\dots,i_j)}} \Bigg(\int_{\D \setminus A_\epsilon} \hat{\V}_\epsilon^{i\sigma}(z_l) \hat{\varphi}_l(z_l) dz_l\Bigg) \\
	&\qquad\times \int_{(\D \setminus A_\epsilon)^{|R_{p,q}^k(i_1,\dots,i_j)|}} U_\epsilon^{(z_{p(1)},z_{q(1)}),\dots,(z_{p(k)},z_{q(k)}),(z_{p(i_1)},z_{q(i_1)}),\dots,(z_{p(i_j)},z_{q(i_j)})} \\
	&\qquad\times \Bigg( \prod_{l\in R_{p,q}^k(i_1,\dots,i_j)} \hat{\V}_\epsilon^{i\sigma}(z_l) \hat{\varphi}_l(z_l) dz_l\Bigg) \Bigg| \F_{\A_{-a,a}}\Bigg]
\end{align*}
where $p,q$ are as above and
\begin{align*}
    R_{p,q}^k(i_1,\dots,i_j) = \{l: l= p(i) \textup{ or } l = q(i) \textup{ for some } i \in \{1,\dots,k,i_1,\dots,i_j\} \}.
\end{align*}
Thus, by the triangle inequality and Jensen's inequality, the absolute value of the conditional expectation of the integral over $B_\epsilon \setminus A_\epsilon^N$ is bounded by a finite sum of terms on the form
\begin{align}\label{eq:Bepsto0}
    &\E\Bigg[ \Bigg| \prod_{\substack{ 1 \leq l \leq N \\ l \notin R_{p,q}^k}} \int_{\D \setminus A_\epsilon} \hat{\V}_\epsilon^{i\sigma}(z_l) \hat{\varphi}_l(z_l) dz_l  \Bigg| \\
    &\times \Bigg|\int_{(\D \setminus A_\epsilon)^{|R_{p,q}^k|}} U_\epsilon^{(z_{p(1)},z_{q(1)}),\dots,(z_{p(k)},z_{q(k)})} \prod_{l\in R_{p,q}^k} \hat{\V}_\epsilon^{i\sigma}(z_l) \hat{\varphi}_l(z_l) dz_l\Bigg| \Bigg| \F_{\A_{-a,a}}\Bigg]. \nonumber
\end{align}
By repeating the argument earlier in this proof, we have that conditional expectation of the product of the integrals on the first line of \eqref{eq:Bepsto0} is bounded as $\epsilon \rightarrow 0$. We will complete the proof by giving a deterministic upper bound on the integral on the second line that will converge to $0$ as $\epsilon \rightarrow 0$, for any choice of $p,q$ and $k$.
Since $|\hat{\V}_\epsilon^{i\sigma}(z_l)| \leq \epsilon^{-\frac{\sigma^2}{2}}$, we have
\begin{align*}
&\Bigg|\int_{(\D \setminus A_\epsilon)^{|R_{p,q}^k|}} U_\epsilon^{(z_{p(1)},z_{q(1)}),\dots,(z_{p(k)},z_{q(k)})} \prod_{l\in R_{p,q}^k} \hat{\V}_\epsilon^{i\sigma}(z_l) \hat{\varphi}_l(z_l) dz_l\Bigg| \\
    &\leq \Bigg( \prod_{l\in R_{p,q}^k} \| \hat{\varphi}_l \|_\infty \Bigg) \epsilon^{-|R_{p,q}^k| \frac{\sigma^2}{2}} \int_{\D^{|R_{p,q}^k|}} U_\epsilon^{(z_{p(1)},z_{q(1)}),\dots,(z_{p(k)},z_{q(k)})} \prod_{l\in R_{p,q}^k} dz_l.
\end{align*}
Next, we estimate the integral 
\begin{align*}
    \int_{\D^{|R_{p,q}^k|}} U_\epsilon^{(z_{p(1)},z_{q(1)}),\dots,(z_{p(k)},z_{q(k)})} \prod_{l\in R_{p,q}^k} dz_l.
\end{align*}
To explain the idea let us first just consider the case $k=1$. Im this case the integral is $\lesssim \epsilon^2$ and if $k=2$, the integral is $\lesssim \epsilon^4$. However, if $k=3$ and $p,q$ are such that $\{z_l: l \in R_{p,q}^3\}$ consists of just three points, that is, we have three points, all within distance $2\epsilon$ from each other, then the best bound we can get on the integral is still $\lesssim \epsilon^4$. This is, because if the first two points are within distance $2\epsilon$, then the third has to be in the intersection of their $2\epsilon$-neighbourhoods, but this region has area $\asymp \epsilon^2$, so we still get a bound $\lesssim \epsilon^4$. The important fact is that since $\sigma <1 \le \sqrt{2}$ the integral decays faster than the factor coming from the estimate on the imaginary chaos blows up, so we have convergence to $0$.

More generally, we let $G_{p,q}^k$ be the graph with vertex set $V_{p,q}^k = \{z_l : l\in R_{p,q}^k\}$ and edge set $E_{p,q}^k=\{ (z_{p(i)},z_{q(i)}): 1 \leq i \leq k\}$ and we let $F_{p,q}^k = (V_{p,q}^k,E_{p,q}^{k,F})$ be a forest constructed by taking one spanning tree on each component of $G_{p,q}^k$. Clearly, $|E_{p,q}^{k,F}| = |V_{p,q}^k| - c_{p,q}^k = |R_{p,q}^k|-c_{p,q}^k$, where $c_{p,q}^k$ is the number of components of $G_{p,q}^k$ (and hence $F_{p,q}^k$). Then, we have that
\begin{align*}
    \int_{\D^{|R_{p,q}^k|}} U_\epsilon^{(z_{p(1)},z_{q(1)}),\dots,(z_{p(k)},z_{q(k)})} \prod_{l\in R_{p,q}^k} dz_l \lesssim \epsilon^{2|E_{p,q}^{k,F}|} = \epsilon^{2(|R_{p,q}^k|-c_{p,q}^k)}.
\end{align*}
Thus, the ``worst'' case is when we have the most possible components, that is, if $|R_{p,q}^k|$ is even and $c_{p,q}^k = |R_{p,q}^k|/2$. Hence we get
\begin{align*}
    &\Bigg|\int_{(\D \setminus A_\epsilon)^{|R_{p,q}^k|}} U_\epsilon^{(z_{p(1)},z_{q(1)}),\dots,(z_{p(k)},z_{q(k)})} \prod_{l\in R_{p,q}^k} \hat{\V}_\epsilon^{i\sigma}(z_l) \hat{\varphi}_l(z_l) dz_l\Bigg| \\
    &\leq \Bigg( \prod_{l\in R_{p,q}^k} \| \hat{\varphi}_l \|_\infty \Bigg) \epsilon^{-|R_{p,q}^k| \frac{\sigma^2}{2}} \int_{\D^{|R_{p,q}^k|}} U_\epsilon^{(z_{p(1)},z_{q(1)}),\dots,(z_{p(k)},z_{q(k)})} \prod_{l\in R_{p,q}^k} dz_l \\
    &\lesssim \epsilon^{2(|R_{p,q}^k|-c_{p,q}^k) - |R_{p,q}^k|\frac{\sigma^2}{2}} \leq \epsilon^{|R_{p,q}^k|(1-\frac{\sigma^2}{2})}
\end{align*}
which converges to $0$ as $\epsilon \rightarrow 0$. Hence, each term of the form \eqref{eq:Bepsto0} converges to $0$ in probability and hence the integral over $B_\epsilon \setminus A_\epsilon^N$ converges in probability to $0$ as $\epsilon \rightarrow 0$, and the proof of Proposition~\\ref{p. general conditioning} is complete. \end{proof}

\subsection{Cosine of the field}
We now analyze $\cos(\sigma \Gamma)$, where $\Gamma$ is a zero-boundary GFF in $\D$. The computations performed here will be a key component in our proof of the two-point estimate for the dimension of two-valued sets. 
	
The calculations are somewhat lengthy but the main idea is simple. A consequence of Proposition \ref{p. general conditioning} is that if $\sigma<\sigma_c$ then the conditional law of $\cos(\sigma \Gamma)$ given $\A_{-a,a}$ is that of $\cos\left(\sigma (\Gamma + h_{\A_{-a,a}}) \right)$, where $\Gamma^{\A_{-a,a}}$ is (somewhat informally) a GFF in $\D\setminus \A_{-a,a}$, and we have the familiar trigonometric formula
\begin{align}\label{e. cos of the sum}
     &  \cos \left( \sigma(\Gamma^{\A_{-a,a}}  + h_{\A_{a,a}}) \right) \nonumber \\& = \cos\left(\sigma \Gamma^{\A_{-a,a}} \right)\cos \left(\sigma h_{\A_{-a,a}}\right)-\sin\left(\sigma\Gamma^{\A_{-a,a}} \right)\sin \left(\sigma h_{\A_{-a,a}} \right).
\end{align}

Let $U$ be a subset compactly contained in $\D$, and set $\overline{\mathcal{V}^{i\sigma}} = \mathcal{V}^{-i\sigma}$. Let
\begin{align}\label{cosU}
	C_U^\sigma \coloneqq (\mathcal{V}^{i\sigma} + \overline{\mathcal{V}^{i\sigma}}, \1_U) = 2(\cos(\sigma \Gamma), \1_U ).
\end{align}

The main observable we will study is then defined by
\begin{equation}\label{e. observable}
    C_U^\sigma C_V^\sigma C_W^\sigma,
\end{equation}
for appropriate compact sets $U,V,W \subset \D$ to be chosen in the next section. We will compute the expected value of this observable in two different ways. We first do this directly, using Lemma \ref{p. n point function}, and then by first conditioning on $\A_{-a,a}$. The first computation is given in the following lemma.
Define
\begin{align*}
	H_D^{\sigma}(x,y,z) &= 2\left( \big<\V^{i \sigma}(x) \V^{i\sigma}(y) \V^{i\sigma}(z)\big>_{D}+ \big<\V^{i \sigma}(x) \V^{i\sigma}(y) \V^{-i\sigma}(z)\big>_{D} \right. \\
	&\left. \qquad +\big<\V^{i \sigma}(x) \V^{i\sigma}(z) \V^{-i\sigma}(y)\big>_{D}+\big<\V^{i \sigma}(y) \V^{i\sigma}(z) \V^{-i\sigma}(x)\big>_{D} \right).
\end{align*}
\begin{lemma}
Let $\sigma<\sigma_c$, and $U,V,W$ be disjoint, measurable sets compactly contained in $\D$. Then
\begin{equation}\label{e.3point}
    \E\left[ C_U^\sigma C_V^\sigma C_W^\sigma\right]= \iiint_{U \times V \times W} H_\D^\sigma(x,y,z)dxdydz.
\end{equation}
\end{lemma}
	\begin{proof}
	This is a special case of Lemma~\ref{p. n point function}.
	\end{proof}

Next, we compute the conditional expectation of \eqref{e. observable}, given $\A_{-a,a}$. Heuristically, one just needs to apply \eqref{e. cos of the sum} and note that sine is an odd function so it will only appear in the expected value if it is squared.
\begin{lemma}
Let $a\geq \lambda$ and $\sigma < \sigma_c$. Let  $U,V,W$ be disjoint, measurable sets compactly contained in $\D$ and set $D_a = (U \setminus \A_{-a,a}) \times (V \setminus \A_{-a,a}) \times (W \setminus \A_{-a,a})$. Then,
\begin{align}\label{condexp}
\E&\left[C_U^\sigma C_V^\sigma  C_W^\sigma\mid \F_{\A_{-a,a}} \right] \nonumber \\
&= 2 \cos(a\sigma)\iiint_{D_a} H_{\D\setminus \A_{-a,a}}^{\sigma}(x,y,z) dx dy dz \\
&-8 \cos(a\sigma)\sin^2(a\sigma) \iiint_{D_a} \Big[ \1_{A_1} \big<\V^{i \sigma}(x) \V^{i\sigma}(y) \V^{i\sigma}(z)\big>_{\D\setminus \Aa} \nonumber \\
&\qquad \qquad \qquad \qquad \qquad \quad + \1_{A_2} \big<\V^{i \sigma}(x) \V^{i\sigma}(y) \V^{-i\sigma}(z)\big>_{\D \setminus \Aa} \nonumber \\
&\qquad \qquad \qquad \qquad \qquad \quad + \1_{A_3} \big<\V^{i \sigma}(x) \V^{i\sigma}(z) \V^{-i\sigma}(y)\big>_{\D \setminus \Aa} \nonumber \\
&\qquad \qquad \qquad \qquad \qquad \quad + \1_{A_4} \big<\V^{i \sigma}(y) \V^{i\sigma}(z) \V^{-i\sigma}(x)\big>_{\D \setminus \Aa} \Big] dx dy dz, \nonumber
\end{align}
where
\begin{align*}
A_1 &= \{(x,y,z) \in D_a: h_{\A_{-a,a}}(x) = h_{\A_{-a,a}}(y) = h_{\A_{-a,a}}(z) \}, \\
A_2 &= \{(x,y,z) \in D_a: h_{\A_{-a,a}}(x) = h_{\A_{-a,a}}(y) \neq h_{\A_{-a,a}}(z) \}, \\
A_3 &= \{(x,y,z) \in D_a: h_{\A_{-a,a}}(x) = h_{\A_{-a,a}}(z) \neq h_{\A_{-a,a}}(y) \}, \\
A_4 &= \{(x,y,z) \in D_a: h_{\A_{-a,a}}(y) = h_{\A_{-a,a}}(z) \neq h_{\A_{-a,a}}(x) \}.
\end{align*}
\end{lemma}
\begin{proof}
We let $S(u,v) = \sin(\sigma h_a(u)) \sin(\sigma h_a(v))$ and define
\begin{align*}
S_{3,0}(u,v,w) &= S(u,v) + S(u,w) + S(v,w), \\
S_{2,1}(u,v,w) &= S(u,v) - S(u,w) - S(v,w).
\end{align*} 
By Proposition \ref{p. general conditioning} and a computation, we have that
\begin{align}\label{e. conditional 3 points}
    \E&\left[C_U^\sigma C_V^\sigma  C_W^\sigma\mid \F_{\A_{-a,a}} \right] \\
    &= 2 \cos(a\sigma)^3\iiint_{D_a} H_{\D\setminus \A_{-a,a}}^{\sigma}(x,y,z) dx dy dz \nonumber \\
    &-2\cos(a\sigma) \iiint_{D_a} \Big[ S_{3,0}(x,y,z) \big<\V^{i \sigma}(x) \V^{i\sigma}(y) \V^{i\sigma}(z)\big>_{\D\setminus \Aa} \nonumber \\
    &\qquad \qquad \qquad \quad \ + S_{2,1}(x,y,z) \big<\V^{i \sigma}(x) \V^{i\sigma}(y) \V^{-i\sigma}(z)\big>_{\D \setminus \Aa} \nonumber \\
    &\qquad \qquad \qquad \quad \ + S_{2,1}(x,z,y) \big<\V^{i \sigma}(x) \V^{i\sigma}(z) \V^{-i\sigma}(y)\big>_{\D \setminus \Aa} \nonumber \\
    &\qquad \qquad \qquad \quad \ + S_{2,1}(y,z,x) \big<\V^{i \sigma}(y) \V^{i\sigma}(z) \V^{-i\sigma}(x)\big>_{\D \setminus \Aa}\Big] dx dy dz, \nonumber
\end{align} 

Note that for every triple $(x,y,z) \in D_a$, exactly one out of $S_{3,0}(x,y,x)$, $S_{2,1}(x,y,z)$, $S_{2,1}(x,z,y)$ and $S_{2,1}(y,z,x)$ takes the value $3\sin^2(a\sigma)$ and the rest take the value $-\sin^2(a\sigma)$. Thus using that $\cos^3(a\sigma)=\cos(a\sigma)(1-\sin^2(a\sigma))$, \eqref{condexp} follows from \eqref{e. conditional 3 points}.
\end{proof}

Next, we will prove the following lower bound for $\E\left[C_U^\sigma C_V^\sigma  C_W^\sigma\mid \F_{\A_{-a,a}} \right]$ which is one of the main inputs in the two-point estimate.

\begin{lemma}\label{l. lower bound observable}
Let $a\geq \lambda$ and $\sigma < \sigma_c$. Let  $U,V,W$ be disjoint, measurable sets compactly contained in $\D$ and set $D_a = (U \setminus \A_{-a,a}) \times (V \setminus \A_{-a,a}) \times (W \setminus \A_{-a,a})$. Then
\begin{align}\label{condexpest}
\E\left[C_U^\sigma C_V^\sigma  C_W^\sigma\mid \F_{\A_{-a,a}} \right] \geq 2 \cos^3(a\sigma)\iiint_{D_a} H_{\D\setminus \A_{-a,a}}^{\sigma}(x,y,z) dx dy dz.
\end{align}
\end{lemma}
\begin{proof}
First, we subdivide $D_a$ into
\begin{align*}
U_1 = \{(x,y,z) \in D_a: &\, x,y, z \text{ belong to the same }   \\
&\, \text{connected component of } \D\setminus \A_{-a,a}\}, \\
U_2 = \{(x,y,z) \in D_a: &\, x, y \text{ belong to the same} \\
&\, \text{connected component of } \D\setminus \A_{-a,a}, z \text{ does not} \}, \\
U_3 = \{(x,y,z) \in D_a: &\, x, z \text{ belong to the same} \\
&\, \text{connected component of } \D\setminus \A_{-a,a}, y \text{ does not} \}, \\
U_4 = \{(x,y,z) \in D_a: &\, y,z \text{ belong to the same} \\
&\, \text{connected component of } \D\setminus \A_{-a,a}, x \text{ does not} \}, \\
U_5 = \{ (x,y,z) \in D_a: &\, x,y,z \text{ belong to different components of } \D\setminus \A_{-a,a} \}.
\end{align*}
Note that $U_1 \subset A_1$, $U_2 \subset A_1 \cup A_2$, $U_3 \subset A_1 \cup A_3$, $U_4 \subset A_1 \cup A_4$ and $U_5 \subset A_1 \cup A_2 \cup A_3 \cup A_4$. Furthermore, on $U_2$, we have
\begin{align*}
    \big<\V^{i \sigma}(x) \V^{i\sigma}(y) \V^{i\sigma}(z)\big>_{\D\setminus \Aa} &= \big<\V^{i \sigma}(x) \V^{i\sigma}(y) \V^{-i\sigma}(z)\big>_{\D \setminus \Aa}, \\
    \big<\V^{i \sigma}(x) \V^{i\sigma}(z) \V^{-i\sigma}(y)\big>_{\D \setminus \Aa} &= \big<\V^{i \sigma}(y) \V^{i\sigma}(z) \V^{-i\sigma}(x)\big>_{\D \setminus \Aa},
\end{align*}
on $U_3$, we have
\begin{align*}
    \big<\V^{i \sigma}(x) \V^{i\sigma}(y) \V^{i\sigma}(z)\big>_{\D\setminus \Aa} &= \big<\V^{i \sigma}(x) \V^{i\sigma}(z) \V^{-i\sigma}(y)\big>_{\D \setminus \Aa}, \\
    \big<\V^{i \sigma}(x) \V^{i\sigma}(y) \V^{-i\sigma}(z)\big>_{\D \setminus \Aa} &= \big<\V^{i \sigma}(y) \V^{i\sigma}(z) \V^{-i\sigma}(x)\big>_{\D \setminus \Aa},
\end{align*}
on $U_4$, we have
\begin{align*}
    \big<\V^{i \sigma}(x) \V^{i\sigma}(y) \V^{i\sigma}(z)\big>_{\D\setminus \Aa} &= \big<\V^{i \sigma}(y) \V^{i\sigma}(z) \V^{-i\sigma}(x)\big>_{\D \setminus \Aa}, \\
    \big<\V^{i \sigma}(x) \V^{i\sigma}(y) \V^{-i\sigma}(z)\big>_{\D \setminus \Aa} &= \big<\V^{i \sigma}(x) \V^{i\sigma}(z) \V^{-i\sigma}(y)\big>_{\D \setminus \Aa},
\end{align*}
and on $U_5$, we have
\begin{align*}
    &\big<\V^{i \sigma}(x) \V^{i\sigma}(y) \V^{i\sigma}(z)\big>_{\D\setminus \Aa} = \big<\V^{i \sigma}(x) \V^{i\sigma}(y) \V^{-i\sigma}(z)\big>_{\D \setminus \Aa} \\
    &= \big<\V^{i \sigma}(x) \V^{i\sigma}(z) \V^{-i\sigma}(y)\big>_{\D \setminus \Aa} = \big<\V^{i \sigma}(y) \V^{i\sigma}(z) \V^{-i\sigma}(x)\big>_{\D \setminus \Aa}.
\end{align*}

We are done if we prove that the integrals in \eqref{condexp}, over the set $U_j$ is lower bounded by 
\begin{align*}
2\cos^3(a\sigma) \iiint_{U_j}  H_{\D\setminus \A_{-a,a}}^{\sigma}(x,y,z) dx dy dz
\end{align*}
for each $j$. We show it for $j=2$. Since $U_2 \subset A_1 \cup A_2$ and since \newline $\big<\V^{i \sigma}(x) \V^{i\sigma}(y) \V^{i\sigma}(z)\big>_{\D\setminus \Aa} = \big<\V^{i \sigma}(x) \V^{i\sigma}(y) \V^{-i\sigma}(z)\big>_{\D \setminus \Aa}$ on $U_2$, we have that the integral in \eqref{condexp} over $U_2$ is
\begin{align*}
&2\cos(a\sigma) \iiint_{U_2} \Big( H_{\D\setminus \A_{-a,a}}^{\sigma}(x,y,z) \\
&\qquad -4\sin^2(a\sigma) \left[ \1_{A_1} \big<\V^{i \sigma}(x) \V^{i\sigma}(y) \V^{i\sigma}(z)\big>_{\D\setminus \Aa} \right. \\
&\qquad \qquad \qquad \qquad \left. + \1_{A_2} \big<\V^{i \sigma}(x) \V^{i\sigma}(y) \V^{-i\sigma}(z)\big>_{\D \setminus \Aa}\right]\Big) dx dy dz \\
&= 2\cos(a\sigma) \iiint_{U_2}  \Big( H_{\D\setminus \A_{-a,a}}^{\sigma}(x,y,z) \\
&\qquad \qquad \qquad \qquad -4\sin^2(a\sigma) \big<\V^{i \sigma}(x) \V^{i\sigma}(y) \V^{i\sigma}(z)\big>_{\D\setminus \Aa}\Big) dx dy dz \\
&\geq 2\cos^3(a\sigma)\iiint_{U_2} H_{\D\setminus \A_{-a,a}}^{\sigma}(x,y,z) dx dy dz,
\end{align*}
where, in the last inequality, we used that $$\big<\V^{i \sigma}(x) \V^{i\sigma}(y) \V^{i\sigma}(z)\big>_{\D\setminus \Aa} \leq \frac{1}{4} H_{\D\setminus \A_{-a,a}}^{\sigma}(x,y,z).$$

The exact same procedure works for $U_3$ and $U_4$. In the case of $U_1$, it is easier, since $\1_{A_1}$ is the only nonzero indicator and in the case of $U_5$, it is just as easy, as the correlation functions are all equal. Thus,
\begin{align*}
\E\left[C_U^\sigma C_V^\sigma  C_W^\sigma\mid \F_{\A_{-a,a}} \right] \geq 2\cos^3(a\sigma) \iiint_{D_a}  H_{\D\setminus \A_{-a,a}}^{\sigma}(x,y,z) dx dy dz,
\end{align*}
and we are done.
\end{proof}

\section{One- and two-point estimates}
This section proves the probabilistic estimates needed for Theorem \ref{thm:dim}. Throughout, we let
\begin{align*}
	d=2 - \frac{2\lambda^2}{(a+b)^2}.
\end{align*}

This section is divided in two parts. We first derive an up-to-constants one-point estimate by studying the law of conformal radius of $\A_{-a,b}$. In the second part, we obtain an upper bound for the two-point estimate using an observable constructed from the cosine of the GFF. 
\subsection{One-point estimate}
We will use Proposition~\ref{cledesc2} for the one-point estimate, in particular we know that $\log r_\D(z) -\log r_{\D \setminus  \A_{-a,b}}(z) $ follows the law of the first time a Brownian motion started from $0$ exists $[-a,b]$.
\begin{lemma}\label{l.1point}Fix $a,b > 0$ such that $a+ b \ge 2\lambda$. There exists $\rho > 0$ such that for all $z$ of distance at least $2\epsilon$ from $\partial \D$,
\begin{align*}
	\mathbb{P}(r_{\D \setminus \A_{-a,b}}(z) \le \epsilon) = c_*r_\mathbb{D}(z)^{d-2} \epsilon^{2-d}(1+O(\epsilon^\rho)),
\end{align*}
where $c_* = 4\pi^{-1} \sin\left( \pi a/(a+b) \right)$. 
\end{lemma}
\begin{proof}
Let $W_t^x$ be standard Brownian motion started from $x \in [0, L]$, $L:=(a+b)\pi/2\lambda$. Then if $\tau_x = \inf\{t\ge 0: W_t^x \in \{0,(a+b) \pi/2\lambda\} \}$, we have that $ \log r_\mathbb{D}(z) -\log r_{\mathbb{D}\setminus A_{-a,b}}(z)\stackrel{d}{=} \tau_{a\pi/2\lambda}$. Moreover, if $u(t,x) := \mathbb{P}\{\tau_x > t \}$ then $u$ solves the problem
	\[
u_t = \frac{1}{2}u_{xx}, \qquad u(0,x) \equiv 1, \quad (x \in (0 ,L)), \qquad u(t,0) = u(t, L) =0.
	\]
	 With $\lambda_k = k \pi/L = 2k\lambda/(a+b) , k\in \mathbb{N}$, the solution can be written
	\[
	u(t,x) = \sum_{k \ge 1}b_ke^{-\lambda_k^2 t/2} \sin(\lambda_k x),  \quad b_k=\frac{2}{L} \int_0^{L} \sin(\lambda_k x) dx.
	\]
	Hence for $\rho = (\lambda_2^2-\lambda_1^2)/2 > 0$, 
	\[
	\mathbb{P}\{\tau_x > t\} =  \frac{4}{\pi} \sin\left(\frac{a \pi}{a+b} \right)e^{-(2-d)t}[1+O(e^{-\rho t})].
	\]
\end{proof}	
\subsection{Two-point estimate}
Let us first sketch the idea for this estimate. Take $x,y\in \widetilde Q$, $0<\delta<|x-y|/4$ and, recalling \eqref{cosU}, define
\begin{align*}
C_x^\sigma= C_{B(x,\delta)}^\sigma, \hspace{0.03 \textwidth}	C_y^\sigma= C_{B(y,\delta)}^\sigma, \textup{ and } C_{x,y}^\sigma=  C_A^\sigma,
\end{align*}
where $A=B(x,4|x-y|)\setminus B(x,3|x-y|)$. 

Let us now consider the following two-point observable \begin{equation}\label{e.observable}
    C_x^\sigma C_y^\sigma C_{x,y}^\sigma.
\end{equation}
We will show that this observable has a small mean. However, on the event that $\A_{-a,a}$ gets close to both $x$ and $y$ this observable becomes large (due to Lemma \ref{l. lower bound observable}). The two-point estimate is obtained by quantifying this idea.
\begin{prop}\label{p. 2 point function}
If  $\sigma< \sigma_c$, then for all $x,y \in \widetilde{Q}$, and sufficiently small $\delta>0$, there is a constant, $K$, such that 
\begin{align*}
\P\left( d(x,\A_{-a,a})\leq \delta, d(y,\A_{-a,a})\leq \delta\right)\leq \frac{K}{(\sigma_c-\sigma)^3}\frac{\delta^{\sigma^2}}{|x-y|^{\sigma^2/2}}
\end{align*}
\end{prop}

Let us first estimate the mean of the observable \eqref{e.observable}.
\begin{lemma}
Let $\sigma<\sqrt{2}$, then for all $x,y \in \widetilde{Q}$, we have that
\begin{align}\label{expest}
\E\left[ C_x^\sigma C_y^\sigma C_{x,y}^\sigma\right]\leq K \delta^4|x-y|^{2-\sigma^2}.
\end{align}
\end{lemma}
\begin{proof}
This follows from the fact that there is a constant $\tilde{K}$ such that for $z,w \in \widetilde{Q}$, $-\ln|z-w|-\tilde{K} \leq G_\D(z,w)\leq -\ln|z-w|+\tilde{K}$, which together with trivial bounds on distances between points in $B(x,\delta)$, $B(y,\delta)$ and $N$, implies that
\begin{align*}
H_\D^\sigma(u,v,w) \leq \hat{K} r_\D(u)^{-\sigma^2/2} r_\D(v)^{-\sigma^2/2} r_\D(w)^{-\sigma^2/2}|x-y|^{-\sigma^2}
\end{align*}
for $(u,v,w) \in B(x,\delta)\times B(y,\delta)\times N$, for some constant $\hat{K}$. Lastly, noting that in $B(x,\delta)$, $B(y,\delta)$ and $N$, $r_\D(z)^{-\sigma^2/2} \asymp 1$ (as the distance to $\partial \D$ from each point in these sets is bounded below by some positive constant), the result follows from \eqref{e.3point}.
\end{proof}
	
We are ready to prove Proposition \ref{p. 2 point function}.
\begin{proof}[Proof of Proposition \ref{p. 2 point function}] For distinct $x,y \in \widetilde{Q}$ and $\delta < |x-y|/4$, consider the event 
\begin{align*}
    E = E_\delta(x,y) = \{d(x,\A_{-a,a}) \le \delta , \, d(y,\A_{-a,a})\leq \delta\}. 
\end{align*}
Then for all $z\in N$, $d(z,\A_{-a,a}) \leq 5|x-y|$. By distortion estimates, we have for $x' \in B(x,\delta)$, $y' \in B(y,\delta)$, that $r_{\D \setminus \A_{-a,a}}(x') \lesssim \delta $, $r_{\D \setminus \A_{-a,a}}(y') \lesssim \delta$, and $r_{\D \setminus \A_{-a,a}}(z) \lesssim |x-y|$, almost surely.
Moreover, we have that
\begin{align*}
G_{\D\setminus \A_{-a,a}} (x',y') \leq G_{\C \setminus [0,\infty)} (-2\delta,-5\delta) = G_{\C \setminus [0,\infty)}(-2,-5) \lesssim 1.
\end{align*}
In the same way, $G_{\D\setminus \A_{-a,a}}(x',z) \lesssim 1$ and $G_{\D\setminus \A_{-a,a}}(y',z)\lesssim 1$. Hence on the event $E$,
\begin{align*}
H_{\D\setminus \A_{-a,a}}^\sigma(x',y',z) \gtrsim \delta^{-\sigma^2}|x-y|^{-\sigma^2/2},
\end{align*}
and by \eqref{condexpest},
\begin{align*}
\E\left[C_x^\sigma C_y^\sigma  C_{x,y}^\sigma\mid \F_{\A_{-a,a}} \right]\1_E \gtrsim \cos(a\sigma )^3 \delta^{4-\sigma^2}|x-y|^{2-\sigma^2/2}\1_E.
\end{align*}
Hence, there are a constants $K, K'<\infty$ so that
\begin{align*}
\P\left\{ E \right\}&\leq \P\left( \E\left[C_x^\sigma C_y^\sigma C_{x,y}^\sigma\mid \F_{\A_{-a,a}} \right] \geq K\cos(\sigma a)^3\delta^{4-\sigma^2}|x-y|^{2-\sigma^2/2} \right) \\
&\leq  K'\frac{1}{a(\sigma_c-\sigma)^3} \delta ^{\sigma^2} |x-y|^{-\sigma^2/2},
\end{align*}
where we used Markov's inequality and \eqref{expest}.
\end{proof}

\section{Hausdorff dimension}\label{sect:dim}
In this section we prove Theorem~\ref{thm:dim}.
\subsection{Upper bound on dimension}	
We start by noting that by Lemma \ref{l.1point}, together with the Koebe 1/4 theorem and the Schwarz lemma, we have
\begin{align}
	\P( d(z,\A_{-a,b}) \leq \epsilon) \asymp \epsilon^{2-d},
\end{align}
for all $z$ of distance at least $2\epsilon$ from $\partial \D$.

Next, we shall construct a cover of a compact subset of $\D$, say, $\frac{1}{2} \overline{\D}$, to bound the Hausdorff dimension there. It will be obvious that this construction works for any compact subset, and hence gives the upper bound. First, we let
\begin{align*}
	Q_{m,n}^\epsilon = \left[-\frac{1}{2}+\frac{m}{4} \epsilon,-\frac{1}{2}+\frac{m+1}{4} \epsilon \right) \times \left[-\frac{1}{2}+\frac{n}{4} \epsilon,-\frac{1}{2}+\frac{n+1}{4} \epsilon\right),
\end{align*}
for $0 \leq m,n \leq 4 \epsilon^{-1}$, and denote by $x_{m,n}^\epsilon$ the center of $Q_{m,n}^\epsilon$, i.e., $x_{m,n}^\epsilon = \left(-\frac{1}{2} + \frac{2m+1}{8} \epsilon, -\frac{1}{2} + \frac{2n+1}{8} \epsilon \right)$. Furthermore, we let $E_{m,n}^\epsilon$ be the event \newline $\{ d(x_{m,n}^\epsilon,\A_{-a,b}) \leq \epsilon \}$,
\begin{align*}
\mathcal{N}_\epsilon = \sum_{m,n = 1}^{4 \epsilon^{-1}} \1_{E_{m,n}^\epsilon}
\end{align*}
and
\begin{align*}
Q_\epsilon = \bigcup_{m,n: E_{m,n}^\epsilon \textup{ occurs}} Q_{m,n}^\epsilon.
\end{align*}
Then,
\begin{align}\label{Nbound}
\E[ \mathcal{N}_\epsilon ] = \sum_{m,n = 1}^{4\epsilon^{-1}} \P(d(x_{m,n}^\epsilon,\A_{-a,b}) \leq \epsilon) \lesssim \sum_{m,n = 1}^{4\epsilon^{-1}} \epsilon^{2-d} \lesssim \epsilon^{-d}.
\end{align}
Also, for every $\epsilon > 0$, $Q_\epsilon$ is a cover of $\A_{-a,b} \cap \frac{1}{2} \overline{\D}$, and hence $\A_{-a,b} \cap \frac{1}{2} \overline{\D} \subset \cup_{l \geq k} Q_{e^{-l}}$, for every nonnegative integer $k$. Thus, if we denote by $\mathcal{H}^s$ and $\mathcal{M}_*^s$ the $s$-dimensional Hausdorff measure and lower Minkowski content, respectively, then
\begin{align*}
\mathcal{H}^s\left(\A_{-a,b} \cap \frac{1}{2} \overline{\D} \right) \leq c \mathcal{M}_*^s \left( \cup_{l \geq k} Q_{e^{-l}} \right),
\end{align*}
for some constant $c$ and all $k \geq 1$. Thus,
\begin{align*}
\E\left[\mathcal{H}^s\left(\A_{-a,b} \cap \frac{1}{2} \overline{\D} \right)\right] \leq \lim_{k \rightarrow \infty} c \sum_{l \geq k} \E[\mathcal{N}_{e^{-l}}] e^{-ls} = 0
\end{align*}
for all $s>d$. Thus, for all $s>d$, $\mathcal{H}^s\left(\A_{-a,b} \cap \frac{1}{2} \overline{\D} \right) = 0$ almost surely, and hence $\dimh \A_{-a,b} \cap \frac{1}{2} \overline{\D} \leq d$. Since this construction works for any compact subset of $\D$, intersected with $\A_{-a,b}$, we have proven the following.
\begin{prop}\label{UB}
Let $a,b>0$, $a+b \geq 2\lambda$. Then, almost surely,
\begin{align*}
	\dimh \A_{-a,b} \leq 2 - \frac{2\lambda^2}{(a+b)^2}.
\end{align*}
\end{prop}

\subsection{Lower bound}
In this section, we prove the lower bound on $\dimh \A_{-a,b}$. To prove the lower bound, we use a standard technique: we construct a  measure that lives on the set that for every $s < d$ has finite $s$-dimensional energy. Again, we will do this for $\A_{-a,b}$ intersected with a compact subset of $\D$, this time, \[\widetilde{Q} \coloneqq \{z: | \textup{Re } z|, |\textup{Im }z| \leq 1/30\},\] for the sake of convenience.

Given the two-point estimate Proposition~\ref{p. 2 point function}, constructing the measure and estimating the energy is now a fairly standard argument but we also need to show that the dimension depends only on $a+b$ and that it is constant almost surely; the proofs of these facts use the particular construction of the two-valued set. 
\begin{proof}[Proof of Theorem~\ref{thm:dim}]
We begin by constructing a measure $\mu$, such that the support of $\mu$ is contained in $\A_{-a,a} \cap \widetilde{Q}$ and
\begin{align*}
I_s(\mu) = \iint \frac{d\mu(x) d\mu(y)}{|x-y|^s} < \infty
\end{align*}
for each $s < d$. This will imply that $\dimh \A_{-a,a} \cap \widetilde{Q} \geq d$ with positive probability.

For now, fix $N \in \mathbb{N}$, consider the following subsets of $\widetilde{Q}$
\begin{align*}
Q_{m,n}^N = &\left[-\frac{1}{30} + \frac{m}{15}e^{-N},-\frac{1}{30} + \frac{m+1}{15}e^{-N} \right] \\
&\times \left[-\frac{1}{30} + \frac{n}{15}e^{-N},-\frac{1}{30} + \frac{n+1}{15}e^{-N} \right],
\end{align*}
for $0 \leq m,n \leq e^N-1$, and let $x_{m,n}^N$ denote the midpoint of $Q_{m,n}^N$. Furthermore, we let $E_N(z)$ denote the event $\{ d(z,\A_{-a,a}) \leq e^{-N}\}$ and write
\begin{align*}
	Q_N = \bigcup_{m,n: E_N(x_{m,n}^N) \textup{ occurs}} Q_{m,n}^N.
\end{align*}
Then,
\begin{align*}
	\mathscr{Q} = \bigcap_{k \geq 1} \overline{\bigcup_{N \geq k} Q_N} \subset \A_{-a,a}.
\end{align*}
Next, we define the random measures, $\mu_N$ as
\begin{align*}
\mu_N(A) = \int_A \sum_{m,n=1}^{e^N-1} \frac{\1_{E_N(x_{m,n}^N)}}{\P(E_N(x_{m,n}^N))} \1_{Q_{m,n}^N}(z) dz.
\end{align*}
The goal is to let the measure $\mu$ that we want to construct, be a subsequential limit of the sequence of measures $(\mu_N)_{N \geq 1}$.
Clearly, $\E[\mu_N(\widetilde{Q})] \asymp 1$. Furthermore, we must bound $\E[\mu_N(\widetilde{Q})^2]$ independently of $N$, since then, by the Cauchy-Schwarz inequality,
\begin{align}\label{unifprob}
\P(\mu_N(\widetilde{Q})>0) \geq \frac{\E[\mu_N(\widetilde{Q})]^2}{\E[\mu_N(\widetilde{Q})^2]} \gtrsim 1,
\end{align}
with an implicit constant which is independent of $N$. This implies that the event on which we want to take a subsequence has positive probability. We note that
\begin{align*}
\mu_N(\widetilde{Q})^2 = \iint_{\widetilde{Q} \times \widetilde{Q}} d\mu_N(x) d\mu_N(y) \leq \iint_{\widetilde{Q} \times \widetilde{Q}} \frac{d\mu_N(x) d\mu_N(y)}{|x-y|^s} = I_s(\mu_N),
\end{align*}
for $s>0$, since $|x-y|<1$ for $x,y \in \widetilde{Q}$. Thus it is enough to bound $\E[I_s(\mu_N)]$ for some $s>0$.
We now bound
\begin{align*}
\E[I_{d-\epsilon}(\mu_N)] = \sum_{k,l,m,n=1}^{e^N -1} \frac{\P(E_N(x_{k,l}^N) \cap E_N(x_{m,n}^N))}{\P(E_N(x_{k,l}^N)) \P(E_N(x_{m,n}^N))} \iint_{Q_{k,l}^N \times Q_{m,n}^N} \frac{1}{|x-y|^{d-\epsilon}} dxdy
\end{align*}
for $\epsilon > 0$ small. First, we note that for all $k,l,m,n$,
\begin{align*}
\iint_{Q_{k,l}^N \times Q_{m,n}^N} \frac{1}{|x-y|^s} dxdy \lesssim e^{(s-4)N},
\end{align*}
where the implicit constant is independent of $N$. Thus, for the diagonal terms, we have
\begin{align*}
&\sum_{(k,l) = (m,n)} \frac{1}{\P(E_N(x_{m,n}^N))} \iint_{Q_{m,n}^N \times Q_{m,n}^N} \frac{1}{|x-y|^s} dx dy \\
&\qquad\lesssim e^{2N}e^{(2-d)N} e^{(d-\epsilon-4)N} = e^{-\epsilon N}.
\end{align*}
For the off-diagonal terms, we have
\begin{align*}
&\sum_{(k,l) \neq (m,n)} \frac{\P(E_N(x_{k,l}^N) \cap E_N(x_{m,n}^N))}{\P(E_N(x_{k,l}^N)) \P(E_N(x_{m,n}^N))} \iint_{Q_{k,l}^N \times Q_{m,n}^N} \frac{1}{|x-y|^{d-\epsilon}} dx dy \\
&\lesssim (\sigma_c-\sigma)^{-3} \sum_{(k,l) \neq (m,n)} \frac{e^{-\sigma^2}}{|x_{k,l}^N-x_{m,n}^N|^{\sigma^2/2}} e^{-2(d-2)N} e^{(d-\epsilon-4)N} \\
&= (\sigma_c-\sigma)^{-3} \sum_{(k,l) \neq (m,n)} \frac{1}{|x_{k,l}^N-x_{m,n}^N|^{\sigma^2/2}} e^{-(\sigma^2+d+\epsilon)N},
\end{align*}
which, choosing $\sigma$ large enough (since $d = 2-\sigma_c^2/2$), is bounded by a constant, independent of $N$. Thus, \eqref{unifprob} holds uniformly in $N$. Hence, with positive probability, $\dimh \A_{-a,a} \cap \widetilde{Q} \geq d$.

We now show that this estimate also holds for other choices of $a,b$. Indeed, consider $\A_{-a,b}$ for possibly different $a,b$. Let $\tilde{\Gamma}$ be a GFF with boundary values $(b-a)/2$ and let $\tilde{\A}_{-\tilde{a},\tilde{b}}$ denote the two-valued set of levels $-\tilde{a}$ and $\tilde{b}$ of $\tilde{\Gamma}$. Then $\tilde{\A}_{-(a+b)/2,(a+b)/2} = \A_{-a,b}$ almost surely and since all previous results can be directly generalized to constant boundary condition different from $0$, the conclusions hold for $\tilde{\Gamma}$ as well. Hence, there is a constant $c>0$ such that
\begin{align}\label{eq:dimgeq}
    \P( \dimh \A_{-a,b} \geq d) \geq c.
\end{align}




\begin{figure}[h!]
	\centering
		\includegraphics[width=120mm]{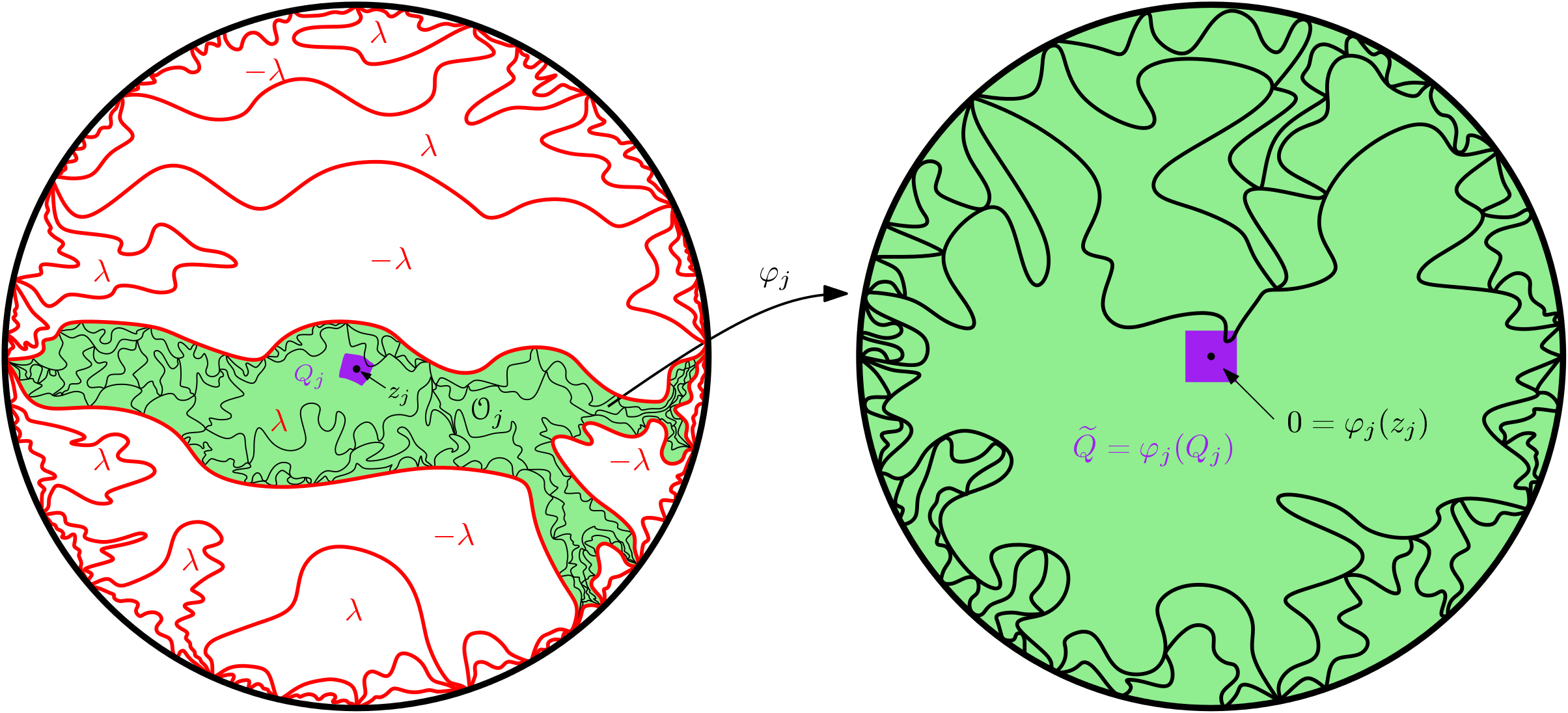}
		\caption{The red set is $\A_{-\lambda,\lambda}$, for each component $\Os_j$ of $\D \setminus \A_{-\lambda,\lambda}$ and the black set inside $\Os_j$ is $\A^j$. We define $\varphi_j:\Os_j \rightarrow \D$ such that $\varphi_j(z_j) = 0$ and $\varphi_j'(z_j) >0$. Then, the law of $\varphi_j(\A^j)$ (black set inside $\D$ in the right picture) is that of a two-valued set with the same parameters, in $\D$.}
		\label{fig:proof01}
	\end{figure}

Now, we turn to proving that the lower bound that we have with positive probability actually holds almost surely and this will conclude the proof of Theorem~\ref{thm:dim}. By the reasoning to get \eqref{eq:dimgeq}, it is enough to consider the local set $\Aa$. As mentioned in the introduction, $\dimh \A_{-\lambda,\lambda} \geq 3/2$ almost surely, since that is the dimension of the $\SLE_4$ curves used to construct $\A_{-\lambda,\lambda}$. The latter is a well-known result due to Beffara \ave{\cite{Bef}}. Thus, assume that $a>\lambda$ and generate $\A_{-\lambda,\lambda}$. Then $\D \setminus \A_{-\lambda,\lambda}$ consists of countably infinitely many connected components, $\{\Os_j\}$, where each $\Os_j$ is a simply connected domain. Conditional on $\A_{-\lambda,\lambda}$, pick a collection of points $\{z_j\}$, such that $z_j \in \Os_j$ and let $\varphi_j$ be the conformal map of $\Os_j$ onto $\D$, such that $\varphi(z_j) = 0$ and $\varphi_j'(z_j) > 0$. Recall that we defined $\widetilde{Q} = \{z: | \textup{Re } z|, |\textup{Im }z| \leq 1/30\}$ and then set $Q_j = \varphi_j^{-1}(\widetilde{Q})$, see Figure \ref{fig:proof01}. Then $Q_j \Subset \varphi_j^{-1}(\frac{1}{2} \D) \Subset \Os_j$ and hence, the restriction $\varphi_j|_{Q_j}$ is bi-Lipschitz.

In each component $\Os_j$ we have an independent zero-boundary GFF plus the harmonic function $h_{\A_{-\lambda,\lambda}} \in \{-\lambda,\lambda\}$. In the components where the harmonic function takes the value $-\lambda$, we explore the set $\A_{-a+\lambda,a+\lambda}$ and in the components where the harmonic function takes the value $\lambda$ we explore $\A_{-a-\lambda,a-\lambda}$. Let $\A^j$ be the two-valued set explored in $\Os_j$. Then $\A_{-\lambda,\lambda} \cup_j \A^j = \Aa$.

We note that for each $j$, we have that if $\A^j = \A_{-a-\lambda,a-\lambda}(\Os_j)$ then $\varphi_j(\A^j)$ has the law of $\A_{-a-\lambda,a-\lambda}(\D)$ and if $\A^j = \A_{-a+\lambda,a+\lambda}(\Os_j)$, then $\varphi_j(\A^j)$ has the law of $\A_{-a+\lambda,a+\lambda}(\D)$. Since $\varphi_j|_{Q_j}$ is bi-Lipschitz, it follows from Theorem 7.5 of \cite{Mat95} that $\dimh \A^j \cap Q_j = \dimh \varphi_j(\A^j) \cap \widetilde{Q}$. Since $\varphi_j(\A^j)$ has the law of either $\A_{-a-\lambda,a-\lambda}(\D)$ or $\A_{-a+\lambda,a+\lambda}(\D)$, \eqref{eq:dimgeq} implies that $\P( \dimh \A^j \cap Q_j \geq d ) \geq c > 0$, for some constant $c$. By the conditional independence of the GFFs in the different regions, this construction works, and gives the same lower bound, for every $j$; that is, $c$ is independent of $j$. It follows that $\dimh \Aa \geq \dimh \cup_j  (\A^j \cap Q_j)  = \sup_j \dimh \A^j \cap Q_j \geq d$, almost surely. Thus, we have that if we fix $a,b>0$ such that $a+b \geq 2\lambda$, then $\dimh \A_{-a,b} = d$ almost surely.

\end{proof}

%
%
\section{Remarks on the conformal Minkowski content}\label{sect:minkowski}
%
In this section, we will consider a particular limiting measure supported on $\A_{-a,b}$.
For $\delta > 0$, define the random measure on $\mathbb{D}$ by the relation
\[
d\mu_\delta =\delta  r_{\D \setminus \A_{-a,b}}(z)^{-(\sigma_c-\delta)^2/2} dz,
\]
where $\sigma_c = 2\lambda/(a+b)$. In what follows, we will write $\nu(f) = \int f d \nu$.
\begin{prop}\label{limitc}
As $\delta \to 0+$, the random measures $\mu_\delta$ converge in law with respect to the weak topology to a random measure $\mu$. In fact, for every
positive, bounded, and measurable function $f$, we have 
\begin{align*}
\mu(f)=\lim_{\delta \to 0+} \mu_\delta(f) \in [0,\infty),
\end{align*}
where the limit is in $L^1(\P)$, and \[\E[\mu(f)]=\frac{2}{a+b}\sin \left (  \frac{\pi a}{a+b} \right)  \int f(z) \, r_\D(z)^{-\sigma_c^2/2} dz .\]
\end{prop}
\begin{proof} Let us fix $a,b>0$ and define $\widetilde V$ as the imaginary chaos related to $\widetilde\Gamma $ which we take to be a GFF with boundary value $(b-a)/2$, i.e., a 0-boundary GFF $\Gamma$ plus the constant $(b-a)/2$. Define $\widetilde \A_{-\tilde a,\tilde a}$ as the two-valued set of level $\tilde a:=(a+b)/2$ of $\widetilde \Gamma$ and note that it is a.s. equal to the two-valued set $\A_{-a,b}$ of $\Gamma$.

Since $\sigma_c \leq 1$, it follows from Proposition 1.4 of \cite{Aru} that $(\widetilde \V^{i\sigma},f)$ converges in $L^1(\P)$ to $(\widetilde\V^{i\sigma_c},f)$ as $\sigma \nearrow \sigma_c$ (note that their proof is for the real multiplicative chaos, but the proof for the imaginary chaos is analogous). Taking real parts, we have
\begin{align*}
\E \left[ \textup{Re}(\widetilde \V^{i\sigma_c},f) \middle| \F_{\widetilde \A_{-\tilde a,\tilde a}} \right] &= \lim_{\delta \to 0+} \E\left[ \textup{Re}(\widetilde \V^{i(\sigma_c-\delta)},f) \middle| \F_{\widetilde \A_{-\tilde a,\tilde a}} \right] \\
&= \lim_{\delta \to 0+} \int f(z) \, r_{\D \setminus \widetilde \A_{-\tilde a,\tilde a}}(z)^{-(\sigma_c-\delta)^2/2} \cos\left((\sigma_c-\delta) \tilde a\right) dz \\
&= \tilde a \lim_{\delta \to 0+} \mu_\delta(f) ,
\end{align*}
where we used that $\cos((\sigma_c-\delta) \tilde a)/\delta \rightarrow \tilde a$ as $\delta \to 0+$ in the last equality and the limit is in $L^1(\P)$. Therefore, for each continuous and bounded $f$, $\mu_\delta(f)$ converges in law to a limiting random variable, so it follows that there exists a limiting random measure $\mu$ such that $\mu_\delta(f) \to \mu(f)$ in law, and where the convergence of the measures is in law with respect to the vague topology. Since the measures we consider are a.s. bounded and since the $\mu_\delta$-mass of the unit disc converges in law, the measures $\mu_\delta$ converge to $\mu$ in law with respect to the weak topology. See, e.g., Chapter~4 of \cite{Kallenberg2017}. 

To obtain the result of the expected value of the measure it is enough to see that
\[ \E\left[ \textup{Re}(\widetilde \V^{i\sigma_c},f) \right]=\cos\left (\sigma_c \frac{b-a}{2}\right ) \int f(z) \, r_\D(z)^{-\sigma_c^2/2} dz, \]
together with
\[ \cos\left (\sigma_c \frac{b-a}{2}\right )=\sin \left (  \frac{\pi a}{a+b} \right)\]
\end{proof}
Let us introduce the following notation, following \cite{ALS1}. Given a function $F: (0,\infty) \rightarrow (0,\infty)$ and a set $A \subset \D$, we define a measure on $\D$ by
\begin{align*}
d\MM_A(F) = \1_{\D \setminus A}(z) F(r_{\D\setminus A}(z))dz.
\end{align*}
Next, for $\delta > 0$, we write
\begin{align*}
	F_\delta(s) = \delta s^{-\sigma_c^2/2 + \delta(\sigma_c-\delta/2)} \1_{\{(0,1)\}}(s)
\end{align*}
and we set $\MM(F) =  \MM_{\A_{-a,b}}(F)$. We claim that Proposition~\ref{limitc} implies that the measures
\[
d\MM(F_\delta) = \delta \1_{\D \setminus \A_{-a,b}}(z)   r_{\D\setminus \A_{-a,b}}(z)^{-\sigma_c^2/2 + \delta(\sigma_c-\delta/2)} \1_{\{(0,1)\}}(r_{\D\setminus \A_{-a,b}}(z))  dz
\]converge in law, with respect to the weak topology, to the random measure $\mu$, which is supported in $\A_{-a,b}$. Indeed, the term $\1_{\D \setminus \A_{-a,b}}(z)$ actually makes no difference when integrating, as $\A_{-a,b}$ has zero Lebesgue measure and that the support is contained in $\A_{-a,b}$ follows since for each $\epsilon > 0$,
\begin{align*}
    \lim_{\delta \rightarrow 0} \delta \int_{d(z,\A_{-a,b}) > \epsilon} f(z) r_{\D \setminus \A_{-a,b}}(z)^{-(\sigma_c-\delta)^2/2} dz = 0,
\end{align*}
almost surely, yet $\mu(f) > 0$ with positive probability. 

We will now argue that \emph{if} the conformal Minkowski content of $\A_{-a,b}$ exists, then it must be equal to a constant times $\mu$.
The conformal Minkowski content of $\A_{-a,b}$ of dimension $d=2-2\lambda^2/(a+b)^2$ is defined by the following limit
\begin{align*}
\M_{-a,b} = \lim_{\eps \rightarrow 0} \epsilon^{d-2} \int_{\D \setminus \A_{-a,b}} \1_{\{r_{\D \setminus \A_{-a,b}}(z) \leq \eps \}} dz.
\end{align*}
The existence of this limit is non-trivial and we do not currently have a proof of it.


This rest of the section is devoted to proving the following conditional proposition, assuming the existence of $\M_{-a,b}$.
\begin{prop}\label{CMC}
Let $a,b>0$ be such that $a+b \geq 2 \lambda$. On the event that $\M_{-a,b}$ exists, we have $\M_{-a,b} \notin \{0,\infty\}$. 
 Moreover, if $\M_{-a,b}$ exists a.s., then
 \begin{align*}
 \E\left[ \M_{-a,b} \right] = c_* \int_\D r_\D(z)^{d-2}dz,
 \end{align*}
 where
 \[
 c_* = \frac{4}{\pi} \sin\left(\frac{\pi a}{a+b} \right).
 \]
\end{prop}
Again, for the sake of convenience, we consider $\Aa$. For $u \in (0,1)$, we let
\begin{align*}
\JJ_u^\sigma(x) = u^{-\frac{\sigma^2}{2}} \1_{\{(0,u)\}}(x)
\end{align*}
and note that if $\M_{-a,a}$ exists, then it is the weak limit of $\int\MM(\JJ_u^{\sigma_c})$ as $u \searrow 0$. 

We now prove the following lemma, which (together with a comment on the case of non-symmetric two-valued set) will give the proof of Proposition~\ref{CMC}.

\begin{lemma}
Let $A \subset \D$ and assume that $\lim_{u \rightarrow 0} \MM_A(\JJ_u^{\sigma_c})$ and $\lim_{\delta \rightarrow 0} \MM_A(F_\delta)$ exist. Then
\begin{align*}
\lim_{u \rightarrow 0} (\MM_A(\JJ_u^{\sigma_c}),f) = \frac{2}{\sigma_c} \lim_{\delta \rightarrow 0} (\MM_A(F_\delta),f)
\end{align*}
for every bounded measurable function $f$.
\end{lemma}
\begin{proof}
Let $\epsilon = \delta(\sigma_c-\delta/2)$ and note that
\begin{align*}
F_\delta(s) = \delta - \int_0^1 F_\delta'(t) \1_{\{s \leq t \}} dt = \delta + \delta\left(\frac{\sigma_c^2}{2}-\epsilon \right)\int_0^1 t^{-\frac{\sigma_c^2}{2}+\epsilon-1} 1_{\{s\leq t\}} dt. 
\end{align*}
Thus, we have
\begin{align*}
(\MM_A(F_\delta),f) &= \delta \int_{\D \setminus A} f(z)dz  \\
&+\delta\left( \frac{\sigma_c^2}{2}-\epsilon\right) \int_{\D \setminus A} f(z) \int_0^1 t^{-\frac{\sigma_c^2}{2}+\epsilon-1} \1_{\{r_{\D \setminus A}(z) \leq t\}} dt dz \\
&= \delta \int_{\D \setminus A} f(z)dz + \delta\left( \frac{\sigma_c^2}{2}-\epsilon\right) \int_0^1 t^{-1+\epsilon} (\MM_A(\JJ_t^{\sigma_c}),f) dt.
\end{align*}
The first term tends to $0$ as $\delta \rightarrow 0$. Making the change of variables $t = e^{-x/\delta}$ in the second integral, we get
\begin{align*}
&\delta\left( \frac{\sigma_c^2}{2}-\epsilon\right)\int_0^1 t^{-1+\epsilon} (\MM_A(\JJ_t^{\sigma_c}),f)dt \\
&= \left( \frac{\sigma_c^2}{2}-\epsilon\right)\int_0^\infty e^{-x(\sigma_c-\frac{\delta}{2})} (\MM_A(\JJ_{\exp(-\frac{x}{\delta})}^{\sigma_c}),f) dx,
\end{align*}
which we want to consider the limit of, as $\delta \rightarrow 0$. 
Since $\lim_{u \rightarrow 0} \MM_A(\JJ_u^{\sigma_c})$ exists, we have that $\sup_{u \in (0,1)} (\MM_A(\JJ_u^{\sigma_c}),f)$ is finite and hence the dominated convergence theorem implies that
\begin{align*}
   &\lim_{\delta \rightarrow 0} \left| \int_0^\infty (e^{-x(\sigma_c-\frac{\delta}{2})} - e^{-x \sigma_c}) (\MM_A(\JJ_{\exp(-\frac{x}{\delta})}^{\sigma_c}),f) dx \right| \\
   &\leq \lim_{\delta \rightarrow 0} \sup_{u \in (0,1)} (\MM_A(\JJ_u^{\sigma_c}),f) \int_0^\infty |e^{-x(\sigma_c-\frac{\delta}{2})} - e^{-x \sigma_c}| dx = 0.
\end{align*}
Thus,
\begin{align*}
&\lim_{\delta \rightarrow 0} (\MM_A(F_\delta),f) \\
&= \lim_{\delta \rightarrow 0} \left( \frac{\sigma_c^2}{2}-\epsilon\right)\int_0^\infty e^{-x(\sigma_c-\frac{\delta}{2})} (\MM_A(\JJ_{\exp(-\frac{x}{\delta})}^{\sigma_c}),f) dx \\
&= \lim_{\delta \rightarrow 0} \frac{\sigma_c^2}{2} \left[ \int_0^\infty (e^{-x(\sigma_c-\frac{\delta}{2})} - e^{-x \sigma_c}) (\MM_A(\JJ_{\exp(-\frac{x}{\delta})}^{\sigma_c}),f) dx \right.\\
&\qquad \qquad \quad \left. + \int_0^\infty e^{-x\sigma_c} (\MM_A(\JJ_{\exp(-\frac{x}{\delta})}^{\sigma_c}),f) dx \right] \\
&=\lim_{u \rightarrow 0} \frac{\sigma_c^2}{2}(\MM_A(\JJ_u^{\sigma_c}),f) \int_0^\infty e^{-x\sigma_c} dx \\
&= \frac{\sigma_c}{2} \lim_{u \rightarrow 0} (\MM_A(\JJ_u^{\sigma_c}),f).
\end{align*}
since the term on the third line vanishes and by using the dominated convergence on the integral on fourth line.

\end{proof}
\begin{proof}[Proof of Proposition~\ref{CMC}]Again, noting that $\A_{-a,b}$ of $\Gamma$ is almost surely equal to $\A_{-(a+b)/2,(a+b)/2}$ of a GFF with different boundary conditions, the non-triviality part of Proposition~\ref{CMC} is proven and hence the proof is complete. 
\end{proof}

\bibliographystyle{plain} 
\bibliography{biblio}

\end{document}